\documentclass[12pt,twoside]{amsart}
\usepackage{amssymb, amsmath}
\usepackage{mathtools}
\usepackage{enumitem}
\usepackage{amsthm}
\usepackage{amsfonts}
\usepackage{tikz-cd}
\usepackage{color}
\usepackage{graphicx}
\usepackage{hyperref} 

\nonstopmode

\numberwithin{equation}{section} 
\font\tengothic=eufm10 scaled\magstep 1
\font\sevengothic=eufm7 scaled\magstep 1
\newfam\gothicfam
      \textfont\gothicfam=\tengothic
      \scriptfont\gothicfam=\sevengothic

\newtheorem*{theorem*}{Theorem}
\newtheorem{theorem}{Theorem}[section]
\newtheorem{lemma}[theorem]{Lemma}
\newtheorem{proposition}[theorem]{Proposition}
\newtheorem{corollary}[theorem]{Corollary}

\theoremstyle{definition} 
\newtheorem{definition}[theorem]{Definition} 
\newtheorem{remark}[theorem]{Remark}
\newtheorem{example}[theorem]{Example}
\newtheorem{notation}[theorem]{Notation}
\newtheorem{question}[theorem]{Question}

\newtheorem{notation and remark}[theorem]{Notation and Remark}
\newtheorem{definition and notation}[theorem]{Definition and Notation}
\newtheorem{observation}[theorem]{Observation}

\newfont{\bg}{cmr10 scaled\magstep4}
\newcommand{\bigzerou}{\smash{\lower1.7ex\hbox{\bg 0}}}

\newcommand\blfootnote[1]{%
  \begingroup
  \renewcommand\thefootnote{}\footnote{#1}%
  \addtocounter{footnote}{-1}%
  \endgroup
}

\begin{document}

\begin{center}
{\Large \bf Componentwise linear syzygies \\[1ex] and good almost regular sequences}
\end{center}
\medskip\medskip\medskip

\begin{center}
{\sc Satoru Isogawa} \\[1ex]
National Institute of Technology, Kumamoto College, \\
Yatsushiro  866-8501, Japan \\
E-mail: isogawa@kumamoto-nct.ac.jp
\end{center}
\medskip\medskip\medskip

\begin{quote}
{\footnotesize 
{\sc Abstract.}\ 
In this paper, we introduce the notion of $\gamma$-regular sequences to characterize the property 
that graded modules have componentwise linear syzygies. 
This extends Harima and Watanabe's characterization of componentwise linear ideals in terms of $\mathfrak{m}$-full property.
}
\end{quote}


%
%

\section{Introduction}
\label{introduction}

In this paper, we study the componentwise linearity of the syzygies of a graded module over a polynomial ring using certain almost regular elements that satisfy a $\it good$ property, which we call $\gamma$-regular elements.
Successive reduction by $\gamma$-regular elements enables us to argue the regularity after reducing the Krull dimension of the base ring.
We denote by $\gamma$-$\mathrm{depth}_{R}M$ the maximal length of successive $\gamma$-regular elements for a module $M$ over a ring $R$.
Our main goal is to prove the following theorem.

\begin{theorem*}\label{Th} 
Let $K$ be an infinite field and $M$ be a finitely generated graded module over the polynomial ring $R=K\lbrack x_1, \cdots, x_n \rbrack $. Then the following are equivalent:
\begin{itemize}
\item[(1)] 
$M$ has a componetwise linear first syzygy.
\item[(2)] 
$\gamma$-$\mathrm{depth}M=n$.
\end{itemize}
\end{theorem*}

This is a generalization of the following theorem in \cite{harima2015completely}.

\begin{theorem}[Harima and Watanabe]\label{Th} 
Let $K$ be an infinite field and $I$ be a homogeneous ideal in the polynomial ring $R=K\lbrack x_1, \cdots, x_n \rbrack $. Then the following are equivalent:
\begin{itemize}
\item[(1)] 
$I$ is a componetwise linear ideal.
\item[(2)] 
$I$ is a completely $\mathfrak{m}$-full ideal.
\end{itemize}
\end{theorem}

The paper is organized as follows. In Section \ref{sec_cwl}, we fix the notation.
From Section \ref{sec_r-reg} to  \ref{sec_rhat_revisited}, we study the properties of $\gamma$-regular sequences,
 and provide the proof of our main theorem in Section \ref{sec_rhat_revisited}. 
 In Section \ref{sec_fdim}, we study finite dimensional modules with linear syzygies.
 In Section \ref{sec_2var}, we study modules with linear syzygies over a polynomial ring in two variables.


The related topics such as Koszul algebra, Koszul modules, and componentwise linear ideals were referred to  \cite{avramov1992regularity}, \cite{conca2010integrally}, \cite{herzog2005koszul}, \cite{iyengar2009linearity}, 
and \cite{lu2013componentwise}.

%
%

\section{Componentwise linear modules}
\label{sec_cwl}

Throughout, let denote $R=K\lbrack x_1, \cdots, x_n \rbrack $ be the polynomial ring over a infinite field $K$ in $n$ variables with the standard grading, 
$\mathfrak{m}=(x_1, \cdots, x_n)$ the graded maximal ideal, $k=R/\mathfrak{m}$ the residue field and $M$ a fintely generated graded module over $R$.  
$\mathbb{Z}$ denotes the set of integers.
\\

For convenience, we introduce several notations below, most of which are standard, with a few exceptions.

\begin{notation}\label{n-1}  Let $N$ be a graded submodule of $M$ and $ z\in R_1\setminus\{0\}$ be a nonzero linear element.
\begin{enumerate}[label=(\arabic*)]
\item 
$ \mathrm{deg}\,M=\mathrm{max}\{j\in\mathbb{Z}\mid M_j\neq0\}$, and in particular, $ \mathrm{deg}\,0=-\infty$.
\item 
$ \mathrm{indeg}\,M=\mathrm{min}\{j\in\mathbb{Z}\mid M_j\neq0\}$, and in particular, $ \mathrm{indeg}\,0=\infty$.
\item 
$ \mathbf{deg}\,M=\{j\in\mathbb{Z}\mid M_j\neq0\}$ the degree support of $M$.
\item
$M_{\langle j \rangle}=RM_j$  the submodule of $M$ generated by elements of degree $j\in\mathbb{Z}$,  and we call $M_{\langle j \rangle}$ the $j$th component of $M$.
\item
$ M_{\geq d}=\bigoplus_{i\geq d}M_i$,  $ M_{\leq d}=\bigoplus_{i\leq d}M_i$ for $ d\in\mathbb{Z}$.
\item
$ N\underset{M}:z=\left\{\xi\in M \mid  z\xi\in N\right\}$.
\item
$\mathrm{p}_{R}\left(M\right)(t)=\sum\limits_{i\geq0}\beta_{i}^R(M)t^i$ the Poincare series of $M$,\\
where $\beta_{i}^R(M)=\dim_{K}\mathrm{Tor}_i^R(M,k)$ the $i$th betti number of $M$.
\item
$\mathrm{P}_{R}\left(M\right)(t,u)=\sum\limits_{i\geq0, j\in\mathbb{Z}}\beta_{ij}^R(M)t^iu^j$ the graded Poincare series of $M$,\\
where $\beta_{ij}^R(M)=\dim_{K}\mathrm{Tor}_i^R(M,k)_{j}$ the graded betti number of $M$.
\item
$\mathrm{reg}_RM=\max\{j-i \mid \beta_{ij}^R(M)\neq 0 \}$ the regularity of $ M$, \\
and in particular, $ \mathrm{reg}_R\,0=-\infty$.
\item
$\mathrm{H}_M(u)=\sum\limits_{j\in\mathbb{Z}}\left(\dim_KM_j\right)u^j$ the Hilbert series of $M$.
\item
$\mathrm{Syz}_1^RM=\mathrm{Ker}(\pi:F_M\to M)$ the first syzygy of $M$, where $\pi$ is a minimal graded free cover of $M$.
\end{enumerate}
\end{notation}

\begin{remark}\label{rem_syz} 
A minimal graded free cover $\pi:F_M\to M$ is unique up to isomorphism. More precisely, let $\pi':F'\to M$ 
be another minimal graded free cover of $M$. 
Then, there exist vertical isomorphisms that make the following diagram, with exact rows, commutative:
$$
\begin{tikzcd}[row sep=0.6cm,column sep=0.6cm]
0\arrow[r]&\mathrm{Syz}_1^RM \arrow[d]\arrow[r] & F_M \arrow[d,,"\theta"]\arrow[r, "\pi"]   & M \arrow[d, equal]\arrow[r] & 0\;\\
0\arrow[r]&\mathrm{Ker}\pi'                      \arrow[r] & F'                          \arrow[r, "\pi' "] & M                          \arrow[r] & 0.
\end{tikzcd}
$$
In fact, we can choose the morphism $\theta$ so that the right square of the above diagram commutes, since $F_M$ is a projective module.
Then, $\theta$ is automatically an isomorphism because, by applying $\underbar{\hspace{7pt}}\otimes k$ to the right square of the above diagram, we see that $\mathrm{Coker}\,\theta=0$ by the graded Nakayama's lemma. Thus, $\theta$ can be viewed as a surjective endomorphism of the Noetherian module $F_M$ since $F_M \simeq F'$. Hence, $\theta$ is an isomorphism.
\end{remark}

\begin{definition}\label{def_cwl}
We say that $M$ is a {\it componentwise linear} module if the following condition holds:
$$\mathrm{reg}_R M_{\langle j \rangle} \leq j \text{ for all } j\in \mathbb{Z}.$$
\end{definition}

\begin{remark}\label{rem_cwl} 
Note that $ \mathrm{reg}_R\,0=-\infty$. We count the zero module as componentwise linear module.
Definition 2.3 is equivalently saying that $M$ is componentwise linear 
if and only if $\mathrm{reg}_R M_{\langle j \rangle} = j$ for all $j\in \mathbb{Z}$ with $ M_{\langle j \rangle} \neq 0$.
\end{remark}
%
%
\section{$\gamma$-regular elements}
\label{sec_r-reg}

In this section, we introduce $\gamma$-regular elements for a finitely generated graded module $M$.
Before going into that, we recall the definition of almost regular elements from Section 1 of \cite{aramova2000almost}.

\begin{definition}\label{def_almost-reg} 
Let $ z\in R_1\setminus\{0\}$  be a nonzero linear element.
We call that $z$ is an {\it almost} $M$-{\it regular} element if $\dim_K\left(0\underset{M}:z\right)<\infty$.
\end{definition}

\begin{notation}\label{n-2}$\quad$
\begin{itemize}
\item[(1)] 
$ \mathrm{NZD} _RM=\{z\in R_1\setminus \{0\}| 0\underset{M}:z=0\}$.
\item[(2)] 
$ \Gamma_{\mathfrak{m}}(M)=\{\xi\in M| \mathfrak{m}^i\xi=0 \text{ for some }i\geq0\}$ 0th local cohomology of $M$.
\item[(3)] 
For each $j\in\mathbb{Z}$, $M(j)$ denotes the $j$-th shift of M, i.e., $$M(j)_i=M_{i+j}\text{ for }i\in\mathbb{Z}.$$
\end{itemize}
\end{notation}

\begin{remark}\label{rem_almost-reg} We note that $\mathrm{NZD}_R\,0=R_1\setminus\{0\}$ by definition and
an element  $z\in R_1\setminus\{0\}$ is almost $M$-regular if and only if $z\in \mathrm{NZD}_R\left(M/\Gamma_{\mathfrak{m}}(M)\right)$.
In fact, applying the snake lemma for the following commutative diagram with exact rows:
$$
\begin{tikzcd}[row sep=0.6cm,column sep=0.6cm]
0\arrow[r]&\Gamma_{\mathfrak{m}}(M) \arrow[d,"\times z"]\arrow[r] & M \arrow[d,"\times z"]\arrow[r] & M/\Gamma_{\mathfrak{m}}(M) \arrow[d, "\times z"]\arrow[r] & 0\;\\
0\arrow[r]&\Gamma_{\mathfrak{m}}(M)(1)                               \arrow[r] & M(1)                                \arrow[r]& \left(M/\Gamma_{\mathfrak{m}}(M)\right)(1)                                \arrow[r] & 0,
\end{tikzcd}
$$
we have the following exact sequence:
$$
0\to 0\underset{\Gamma_{\mathfrak{m}}(M)}:z \to 0\underset{M}:z \to X
\to \left(\Gamma_{\mathfrak{m}}(M)/z\Gamma_{\mathfrak{m}}(M)\right)(1),
$$
where we denote $X= 0\underset{M/\Gamma_{\mathfrak{m}}(M)}:z$. \\
Here, we note that $X=0$ if and only if $z\in \mathrm{NZD}_R\left(M/\Gamma_{\mathfrak{m}}(M)\right)$.
Using the above exact sequence, if $z\in \mathrm{NZD}_R\left(M/\Gamma_{\mathfrak{m}}(M)\right)$, then 
$$0\underset{\Gamma_{\mathfrak{m}}(M)}:z = 0\underset{M}:z <\infty.$$ 
On the other hand, if  $\dim_K\left(0\underset{M}:z\right) <\infty$, 
then $M/\Gamma_{\mathfrak{m}}(M)$ contains $X$ as a submodule of finite length. However, we note that  $ M/\Gamma_{\mathfrak{m}}(M)$ has no non-zero finite length submodule. Thus, we conclude that $X=0$.
\end{remark}

\begin{definition}\label{def_r-reg} 
Let $z\in R_1\setminus\{0\}$  be a nonzero linear element.
We call that $z$ is an $M$-$\gamma$-{\it regular element} or $\gamma$-{\it regular element} for $M$ if the following condition holds:
$$\rho=\rho_R^z(M):\mathrm{Tor}_1^R(M,R/zR) \to \mathrm{Tor}_1^R(M,k) \text{ is injective,}$$
where the map $\rho$ is induced by the canonical surjection $R/zR\to k$.
\end{definition}

\begin{notation}\label{n-3}$\quad$
\begin{itemize}
\item[(1)] 
$\gamma$-$\mathrm{NZD}_RM=\left\{ z \in  R_1\setminus\{0\} \mid z \text{ is an } M\text{-}\gamma\text{-regular element}\right\}$.
\item[(2)] 
$ \mathrm{soc}M=0\underset{M}:\mathfrak{m}=\{\xi\in M| \mathfrak{m}\xi=0\}$ the socle of $M$.
\item[(3)] 
$ \alpha(M;z)=\dim_K\left(0\underset{M}:z\right)$.
\end{itemize}
\end{notation}

If $z\in \gamma\text{-}\mathrm{NZD}_R M$, then $\dim_K\left(0\underset{M}:z\right)<\infty$ 
since 
$$\mathrm{Tor}_1^R(M,R/zR)\simeq \mathrm{Ker}(M(-1)\overset{\times z}{\rightarrow} M) \simeq  \left(0\underset{M}:z\right)(-1).$$
Hence $M$-$\gamma$-regular element is an almost  $M$-regular element.

\begin{remark}\label{rem_r-reg} The following hold:
\begin{itemize}
\item[(1)] 
$\gamma$-$\mathrm{NZD}_R\,0=R_1\setminus\{0\}$ by definition.
\item[(2)] 
If $R=K$, i.e., $n=0$, then $\gamma$-$\mathrm{NZD}_R\,M=\emptyset$ for any graded module $M$  since $R_1=K_1=\emptyset$.
\item[(3)] 
Since $\mathrm{Tor}_1^R(M,k)$ is a $k$-vector space, for $z\in \gamma\text{-}\mathrm{NZD}_R M$, \\
$\rho:\mathrm{Tor}_1^R(M,R/zR) \to \mathrm{Tor}_1^R(M,k)$ is a split injective map.
\item[(4)]
$\mathrm{Tor}_1^R(M,R/zR)\simeq \mathrm{soc}M(-1)$ for $z\in \gamma\text{-}\mathrm{NZD}_R M$.
\item[(5)] 
$\mathrm{reg}_R\, \mathrm{soc}M=\mathrm{deg}\, \mathrm{soc}M$ 
since $\mathrm{reg}_R\,  k(-j)=j$ for $j\in \mathbb{Z}$.
\item[(6)] 
$\mathrm{NZD}_R M\subseteq \gamma\text{-}\mathrm{NZD}_R M\subseteq \mathrm{NZD}_R\left(M/\Gamma_{\mathfrak{m}}(M)\right)$.
\end{itemize}
\end{remark}

The advantage of having $\gamma$-regular elements is that the following splitting lemma holds.

\begin{lemma}\label{lemA_splitting} 
We denote $\overline{R}=R/zR$ for $z \in R_1\setminus\{0\}$.\\
If $z\in \gamma\text{-}\mathrm{NZD}M$, then  the following hold:
\begin{itemize}
\item[(1)] 
$\mathrm{Syz}_1^RM\otimes\overline{R}\simeq\mathrm{Syz}_1^{\overline{R}}\overline{M}\oplus\mathrm{soc}M(-1)$.
\item[(2)] 
$ \beta_{ij}^R\left(\mathrm{Syz}_1^RM\right)= \beta_{ij}^{\overline{R}}\left(\mathrm{Syz}_1^{\overline{R}}\overline{M}\right)+ \beta_{ij}^{\overline{R}}\left( \mathrm{soc}M(-1) \right)$ for $ i\geq0$ and $ j\in \mathbb{Z}$.
\item[(3)] 
$ \mathrm{P}_{R}\left(\mathrm{Syz}_1^RM\right)(t,u)=\mathrm{P}_{\overline{R}}\left(\mathrm{Syz}_1^{\overline{R}}\overline{M}\right)(t,u)+u\mathrm{P}_{\overline{R}}\left(\mathrm{soc}M\right)(t,u)$.
\item[(4)] 
$\mathrm{P}_{R}\left(M\right)(t,u)=\mathrm{P}_{\overline{R}}\left(\overline{M}\right)(t,u)+tu\mathrm{P}_{\overline{R}}\left(\mathrm{soc}M\right)(t,u)$.
\item[(5)] 
$ \mathrm{p}_{R}\left(\mathrm{Syz}_1^RM\right)(t)=\mathrm{p}_{\overline{R}}\left(\mathrm{Syz}_1^{\overline{R}}\overline{M}\right)(t)+\alpha(M;z)(1+t)^{n-1}$.
\item[(6)] 
$ \mathrm{p}_{R}\left(M\right)(t)=\mathrm{p}_{\overline{R}}\left(\overline{M}\right)(t)+\alpha(M;z)t(1+t)^{n-1}$.
\item[(7)] 
$\mathbf{deg}\mathrm(\mathrm{soc}M(-1))\subseteq \mathbf{deg}\,(\mathrm{Syz}_1^RM\otimes k) $.
\item[(8)] 
$ \mathrm{reg}_R\left(\mathrm{Syz}_1^RM\right)=\mathrm{max}\left\{\mathrm{reg}_{\overline{R}}\left(\mathrm{Syz}_1^{\overline{R}}\overline{M}\right),\;\mathrm{deg}\,(\mathrm{Syz}_1^RM\otimes k)\right\}$.
\end{itemize}
\end{lemma}

Before proving Lemma \ref{lemA_splitting} above, we make the following observation:

\begin{observation}\label{obs_lemA} 
Let $ \pi:F_M\to M$ be a minimal graded free cover of $ M$ and denote $\overline{R}=R/zR$ for $z \in R_1\setminus\{0\}$.
There is the following exact sequence:
 $$
\begin{tikzcd}[row sep=0.6cm,column sep=0.6cm]
0\arrow[r]&\mathrm{Tor}_1^R(M,\overline{R}) \arrow[d,"\rho"]\arrow[r,"\iota"] & \mathrm{Syz}_1^RM\otimes \overline{R}\arrow[d,,"\eta"]\arrow[r]   & \mathrm{Syz}_1^{\overline{R}}\overline{M} \arrow[r] & 0\\
               &\mathrm{Tor}_1^R(M,k)                  \arrow[r,"\sim","\theta" ']                                   & \mathrm{Syz}_1^RM\otimes k.
\end{tikzcd}
$$
 which is induced as follows:\\
Applying tensor product  functors $\underline{\hspace{6pt}}\otimes\overline{R}$ and $\underline{\hspace{6pt}}\otimes k$ to the following exact sequence:
      $$ 0\to\mathrm{Syz}_1^RM\to F_M \overset{\pi}\longrightarrow M\to 0,$$
 we obtain the following commutative diagram, where the first row is exact:
  $$
\begin{tikzcd}[row sep=0.6cm,column sep=0.6cm]
0\arrow[r]&\mathrm{Tor}_1^R(M,\overline{R}) \arrow[d,"\rho"]\arrow[r,"\iota"] & \mathrm{Syz}_1^RM\otimes \overline{R}\arrow[d,,"\eta"]\arrow[r] 
& F_M\otimes\overline{R} \arrow[r,"\pi\otimes\overline{R}"] & M\otimes\overline{R} \arrow[r] & 0\;\\
               &\mathrm{Tor}_1^R(M,k)                  \arrow[r,"\theta"]                                   & \mathrm{Syz}_1^RM\otimes k.
\end{tikzcd}
$$
In the above diagram, we note the following:\\
 (i) $ \theta$ is an isomorphism.\\
 (ii) $ \mathrm{Ker}\left(\pi\otimes\overline{R}\right)\simeq \dfrac{\mathrm{Syz}_1^RM+zF_M}{zF_M}\simeq \mathrm{Syz}_1^{\overline{R}}\overline{M}$.
 \end{observation}
 
Using the above Observation \ref{obs_lemA}, we prove Lemma  \ref{lemA_splitting}.

\begin{proof}[Proof of Lemma \ref{lemA_splitting}]
(1) Here, we note that $ \mathrm{Tor}_1^R(M,\overline{R})\simeq\mathrm{soc}M(-1)$ by Remark $\ref{rem_r-reg}$ (4). 
Using the commutative diagram of Observation $\ref{obs_lemA}$, $ \iota$ is split since $ \eta\circ\iota=\theta\circ\rho$ is a split $k\text{-linear}$ map. \\
(2) We note that $ \mathrm{Tor}_i^{R}(\mathrm{Syz}_1^RM,\overline{R})=0$ for all integer $ i\geq1$, since $ z$ is a regular element for  $ \mathrm{Syz}_1^RM\subseteq F_M$.
Using (1),  we have the following:
 $$ \beta_{ij}^R\left(\mathrm{Syz}_1^RM\right)=  \beta_{ij}^{\overline{R}}\left(\mathrm{Syz}_1^{R}M\otimes\overline{R}\right)=\beta_{ij}^{\overline{R}}\left(\mathrm{Syz}_1^{\overline{R}}\overline{M}\right)+ \beta_{ij}^{\overline{R}}\left( \mathrm{soc}M(-1) \right).$$
(3) follows directly from (2). (4) follows from (3) since
 $$
 \mathrm{P}_{\overline{R}}\left(\overline{M}\right)(t,u)=\mathrm{H}_{M\otimes k}(u)+t\mathrm{P}_{\overline{R}}\left(\mathrm{Syz}_1^{\overline{R}}\overline{M}\right)(t,u).
 $$
(5) follows from (3), and (6) follows from (4), respectively since
 $$
 \mathrm{p}_{\overline{R}}\left(\mathrm{soc}M\right)(t)=\alpha(M;z)(1+t)^{n-1}.
 $$
(7) follows from the following inclusion:
$$
\mathrm{soc}M(-1)\simeq \mathrm{Tor}_1^R(M, \overline{R})\hookrightarrow
\mathrm{Tor}_1^R(M, k)\simeq \mathrm{Syz}_1^RM\otimes k
$$
(8) From (2) and  Remark \ref{rem_r-reg} (5),  we have
  $$ \mathrm{reg}_R\left(\mathrm{Syz}_1^RM\right)=\mathrm{max}\left\{\mathrm{reg}_{\overline{R}}\left(\mathrm{Syz}_1^{\overline{R}}\overline{M}\right),\;\mathrm{deg}\mathrm(\mathrm{soc}M(-1))\right\}.$$
Here, we can replace $\mathrm{deg}\mathrm(\mathrm{soc}M(-1))$ by $\mathrm{deg}\,(\mathrm{Syz}_1^RM\otimes k)$ since
 $$
 \mathrm{deg}\mathrm(\mathrm{soc}M(-1)) \leq \mathrm{deg}\,(\mathrm{Syz}_1^RM\otimes k) \leq \mathrm{reg}_R\left(\mathrm{Syz}_1^RM\right).
 $$
\end{proof}

\begin{remark}\label{rem_kos-split}
We consider the following conditions:
\begin{itemize}
\item[(I)] $R$ is a standard graded Noetherian commutative $K$-algebra.
\item[(II)] $R$ is a standard graded Noetherian commutative Koszul $K$-algebra, i.e., $\mathrm{reg}_Rk=0$.
\end{itemize}
The following hold:
\item[(1)] 
Under condition (I),
if $z\in \gamma\text{-}\mathrm{NZD}M\cap\mathrm{NZD}R$, 
then statements (1) to (4) in Lemma \ref{lemA_splitting} remain valid.
\item[(2)]
Under condition (I),
if $z\in \mathrm{NZD}R$, then  $z\in \gamma\text{-}\mathrm{NZD}k$, and
$R$ is a Koszul algebra if and only if $\overline{R}=R/zR$ is a Koszul algebra, 
since the following holds:
$$
\mathrm{P}_{R}\left(k\right)(t,u)=
\mathrm{P}_{\overline{R}}\left(k\right)(t,u)+tu\mathrm{P}_{\overline{R}}\left(\mathrm{soc}k\right)(t,u)
=(1+tu)\mathrm{P}_{\overline{R}}\left(k\right)(t,u)
$$
from (1) and Lemma \ref{lemA_splitting} (4).
\item[(3)]
Under condition (II),
if $z\in \gamma\text{-}\mathrm{NZD}M\cap\mathrm{NZD}R$, 
then statements  (1) to (4) and (8) in Lemma \ref{lemA_splitting} remain valid from (1) and (2).
\end{remark}

As a consequence of the preceding Lemma  \ref{lemA_splitting} (1), we have the following proposition:

\begin{proposition}\label{propA_splitting} 
We denote $\overline{R}=R/zR$ and $\overline{M}=M\otimes \overline{R}$ for $z \in R_1\setminus\{0\}$.
If $z\in \gamma\text{-}\mathrm{NZD}M$, then  there is the following quasi-isomorphism of complexes:
$$ 
M \overset{\mathbb{L}}{\underset{R}\otimes} \overline{R}\underset{\text{qis.}}\simeq \overline{M}\lbrack0\rbrack\oplus \mathrm{soc}M(-1)\lbrack-1\rbrack, 
$$
where $ \_\overset{\mathbb{L}}\otimes_R\_$ denotes the right derived tensor product,  
and $ M\lbrack -i\rbrack$ denotes the stalk complex concentrated in homological degree $ i$ for $ i=0,1$.
\end{proposition}

\begin{proof}
By Lemma $\ref{lemA_splitting}$ (1), we note that there is a splitting epimorphism 
$$ \overline{\mathrm{Syz}_1^RM}=\mathrm{Syz}_1^RM \otimes \overline{R}\to\mathrm{soc}M(-1).
$$
This induces a morphism of complexes 
$$ f:\overline{M}\lbrack0\rbrack\oplus \overline{\mathrm{Syz}_1^RM}\lbrack-1\rbrack\to\overline{M}\lbrack0\rbrack\oplus \mathrm{soc}M(-1)\lbrack-1\rbrack.
$$
On the other hand, there is the natural morphism of complexes 
$$ 
g:F_{*}\underset{R}\otimes \overline{R}\to\overline{M}\lbrack0\rbrack\oplus \overline{\mathrm{Syz}_1^RM}\lbrack-1\rbrack, 
$$
where $ F_{*}$ is a minimal graded free resolution of $M$. Therefore we have the following 
quasi-isomorphisms:
$$
M \overset{\mathbb{L}}{\underset{R}\otimes} \overline{R} \underset{\text{qis.}}\simeq F_{*}\underset{R}\otimes \overline{R} \overset{f\circ g}\longrightarrow\overline{M}\lbrack0\rbrack\oplus \mathrm{soc}M(-1)\lbrack-1\rbrack. 
$$ 
\end{proof}

\begin{question} \label{ques_lemA-splitting}
Can we conclude that $R_1\setminus\{0\}\ni z$ is an $M$-$\gamma$-regular elemnt if the following injective map in Observation \ref{obs_lemA} is split?
$$
\mathrm{Tor}_1^R(M,R/zR) \overset{\iota}\to \mathrm{Syz}_1^RM\otimes R/zR.
$$
\end{question}

%
%
\section{$\gamma$-regular sequences}
\label{sec_r-seq}

In this section, we introduce $\gamma$-regular sequences and describe their properties.

\begin{definition}\label{def_r-seq}
Let $ \underline z=z_1,\ldots,z_r\in R_1\setminus\{0\}$  be a sequence of nonzero linear elements.
We denote
$\overline R^{(i)}=R/(z_1,\ldots,z_i)R,\;\overline M^{(i)}=M\otimes\overline R^{(i)} $ for  $i=1,\ldots,r$.
We say that $\underline z$ is an $M\text{-}\gamma\textit{-regular sequence}$ of length $r$ if the following conditions hold: 
$$
\overline{z_i}\in \gamma\text{-}\mathrm{NZD}_{\overline R^{(i-1)}}\overline M^{(i)},
$$
where $\overline{z_i}$ is a image of $z_i$ in $\overline R^{(i-1)}$ for $i=1,\ldots,r$.
\end{definition}

\begin{notation}We use the following notation:
\begin{itemize}
\item[(1)]
$\gamma\text{-}\mathrm{depth}_R M$: The maximal length of $M$-$\gamma$-regular sequence.
\item[(2)]
$\mathrm{depth}_R M$: The maximal length of $M$-regular sequence.
\item[(3)]
$\mathrm{pd}_R M$: The projective dimension of $M$.
\end{itemize}
\end{notation}

\begin{remark}\label{rem_r-seq}The following hold:
\begin{enumerate}[label=(\arabic*)]
\item
If $\mathrm{depth}_R M\geq1$, then 
$\mathrm{NZD}_R M = \gamma\text{-}\mathrm{NZD}_R M \neq \emptyset$ 
by Remark \ref{rem_r-reg} (6).
\item
$\mathrm{depth}_R M\leq \gamma\text{-}\mathrm{depth}_R M$.
\item
If $M$ is a free module, then any linearly independent sequence $ \underline z=z_1,\ldots,z_n\in R_1\setminus\{0\}$ is an $M$-$\gamma$-regular sequence, and $\gamma\text{-}\mathrm{depth}_R M=\mathrm{depth}_R M=n$.
\item
If $M$ is a module with $\mathfrak{m}M=0$, then any linearly independent sequence $ \underline z=z_1,\ldots,z_n\in R_1\setminus\{0\}$ is an $M$-$\gamma$-regular sequence, since $M$ is isomorphic to a direct sum of copies of the residue field $k$ up to shifting, and $\gamma\text{-}\mathrm{depth}_R M=n$ whereas $\mathrm{depth}_R M=0$.
\end{enumerate}
\end{remark}

Using Lemma \ref{lemA_splitting} repeatedly, we obtain the following lemma:

\begin{lemma}\label{lem_r-seq}
Let $ z_1,\ldots,z_r$ be an $ M\text{-}\gamma\text{-}$ sequence. We denote $\overline{R}^{(i)}=R/(z_1,\ldots,z_i)R$, $\overline{M}^{(i)}=M\otimes \overline{R}^{(i)}$ for $i=1,\cdots, r$ and $\overline{M}^{(0)}=M$.  Then the following hold:
\begin{itemize}
\item[(1)]
$\overline{\mathrm{Syz}_1^R M}^{(i)}\simeq\mathrm{Syz}_1^{\overline{R}^{(i)}} \overline{M}^{(i)}\oplus {\displaystyle\bigoplus_{j=1}^{i}}\, \mathrm{soc}\overline{M}^{(i-1)}(-1)$ for $i=1,\cdots, r$.
\item[(2)] 
$ \mathrm{reg}_R\left(\mathrm{Syz}_1^RM\right)=\mathrm{max}\left\{\mathrm{reg}_{\overline{R}^{(r)}}\left(\mathrm{Syz}_1^{\overline{R}^{(r)}}\overline{M}^{(r)}\right),\;\mathrm{deg}\,(\mathrm{Syz}_1^RM\otimes k)\right\}.$
\item[(3)] $ \mathrm{p}_{R}\left( \mathrm{Syz}_1^RM \right)(t)
=\mathrm{p}_{\overline{R}^{(r)}}\left(\mathrm{Syz}_1^{\overline{R}^{(r)}}\overline{M}^{(r)}\right)(t)
+\sum\limits_{i=1}^r\alpha\left(\overline{M}^{(i-1)};z_i\right)(1+t)^{n-i}$.
\item[(4)] $ \mathrm{p}_{R}\left( M \right)(t)=\mathrm{p}_{\overline{R}^{(r)}}\left(\overline{M}^{(r)}\right)(t)+\sum\limits_{i=1}^r\alpha\left(\overline{M}^{(i-1)};z_i\right)t(1+t)^{n-i}$.
\end{itemize}
\end{lemma}

\begin{proof} (1) follows from Lemma \ref{lemA_splitting} (1).
(2) follows from Lemma \ref{lemA_splitting} (8), since the following hold:
$$ 
\mathrm{deg}\left(\mathrm{Syz}_1^{\overline{R}^{(i)}}\overline{M}^{(i)}\otimes k\right) 
\leq \mathrm{deg}\left(\mathrm{Syz}_1^{R}M\otimes k\right)\text{ for }i=1,\cdots, r-1.
$$
(3) follows from Lemma \ref{lemA_splitting} (5).
(4) follows from Lemma \ref{lemA_splitting} (6).
\end{proof}

As a corollary of Lemma \ref{lem_r-seq},  we obtain the following results.

\begin{corollary}\label{cor_r-seq-pred}
Let $ z_1,\ldots,z_r$ be an $M\text{-}\gamma\text{-}$ sequence. We denote $\overline{R}^{(i)}=R/(z_1,\ldots,z_i)R$, $\overline{M}^{(i)}=M\otimes \overline{R}^{(i)}$ for $i=1,\cdots, r$ and $\overline{M}^{(0)}=M$.
If $\mathrm{pd}_{\overline{R}^{(r)}} \overline{M}^{(r)}=0$, then the following hold:
\begin{itemize}
\item[(1)] 
$ \mathrm{reg}_R\left(\mathrm{Syz}_1^RM\right)=\mathrm{deg}\,(\mathrm{Syz}_1^RM\otimes k).$
\item[(2)] $ \mathrm{p}_{R}\left( M \right)(t)=\dim_K M\otimes k+\sum\limits_{i=1}^r\alpha\left(\overline{M}^{(i-1)};z_i\right)t(1+t)^{n-i}$.
\item[(3)]
$\beta_1^R(M)=\sum\limits_{i=1}^r\alpha\left(\overline{M}^{(i-1)};z_i\right)$ and 
$\beta_2^R(M)=\sum\limits_{i=1}^r(n-i)\alpha\left(\overline{M}^{(i-1)};z_i\right)$.
\end{itemize}
\end{corollary}

\begin{remark}\label{rem_kos-reg}
Assume that $R$ is a standard graded Noetherian commutative Koszul $K$-algebra with $\dim_K\mathfrak{m}\otimes k =n$.
If  $z_1,\ldots,z_r$ form an $M\text{-}\gamma\text{-}$ sequence and an $R$-regular sequence,
then statements (1) and (2) in Lemma \ref{lem_r-seq}, and statements (1) and (3) in Corollary \ref{cor_r-seq-pred}
remain valid by Remark \ref{rem_kos-split}.
\end{remark}

From Lemma $\ref{lem_r-seq}$ (3), 
we derive the following remark.

\begin{remark}\label{rem_alpha}
The values $\alpha\left(\overline{M}^{(i-1)};z_i\right)$ for $i=1,\cdots, r$ are 
independent of the choice of  the $M\text{-}\gamma\text{-}$ sequence $ z_1,\ldots,z_r$.
In fact, $\mathrm{p}_{R}\left( \mathrm{Syz}_1^RM \right)(t)$
can be expressed in the following form:
$$
\mathrm{p}_{R}\left( \mathrm{Syz}_1^RM \right)(t)=\sum\limits_{i=1}^nc_i(1+t)^{n-i}.
$$
Hence,  we obtain $\alpha\left(\overline{M}^{(i-1)};z_i\right)=c_i$ 
for $i=1,\cdots, r$,
since the degree of the polynomial 
$\mathrm{p}_{\overline{R}^{(r)}}\left(\mathrm{Syz}_1^{\overline{R}^{(r)}}\overline{M}^{(r)}\right)(t)$ 
in Lemma $\ref{lem_r-seq}$ (3)
is less than $n-r$,  \\provided that
$\mathrm{p}_{\overline{R}^{(r)}}\left(\mathrm{Syz}_1^{\overline{R}^{(r)}}\overline{M}^{(r)}\right)(t)\neq0$.
\end{remark}

\begin{lemma}\label{lem_r-seq-syz}The follwoing hold:
\begin{itemize}
\item[(1)] If $ \underline{z}=z_1,\ldots,z_r$ is an $M\text{-}\gamma\text{-}$regular sequence,
 then the following hold:
 \begin{itemize}
\item[(i)]  
$\gamma\text{-}\mathrm{NZD}_{\overline{R}^{(i)}}\overline{\mathrm{Syz}_1^R M}^{(i)}=\overline{R}^{(i)}_1\setminus\{0\}$ for $i=1,\cdots, r$.
\item[(ii)]
$ \underline{z}$ is also a $ \mathrm{Syz}_1^RM\text{-}\gamma\text{-}$regular sequence.
\end{itemize}
\item[(2)] $ \gamma\text{-}\mathrm{depth}_R\, \mathrm{Syz}_1^RM \geq 
\mathrm{min}\{n,\gamma\text{-}\mathrm{depth}_R\, M +1\}$.
\end{itemize}
\end{lemma}

\begin{proof} (1) (i) follows from Lemma \ref{lem_r-seq} (1). 
\\(1) (ii) follows from the following:
 \begin{itemize}
\item[(a)] 
$z_1\in \mathrm{NZD}_R\mathrm{Syz}_1^RM \subseteq \gamma\text{-}\mathrm{NZD}_R\mathrm{Syz}_1^RM$ by Remark \ref{rem_r-reg} (6).
\item[(b)]
$\overline{z_i} \in \gamma\text{-}\mathrm{NZD}_{\overline{R}^{(i-1)}}\overline{\mathrm{Syz}_1^R M}^{(i-1)}=\overline{R}^{(i-1)}_1\setminus\{0\}$ for $i=2,\cdots, r$ by (1) (i).
\end{itemize}
(2) Let $\gamma\text{-}\mathrm{depth}_R\,M=r$. If $r=n$, then $\gamma\text{-}\mathrm{depth}_R\, \mathrm{Syz}_1^RM \geq n$ by (1) (ii). Assume that $r<n$ and 
$\underline{z}=z_1,\ldots,z_r$ is an $M\text{-}\gamma\text{-}$regular sequence. Then $ \underline{z}$ is also a $ \mathrm{Syz}_1^RM\text{-}\gamma\text{-}$regular sequence by (1) (ii) and we can choose an element $z_{r+1}\in R_1\setminus\{0\}$ such that
$$
\overline{z_{r+1}}\in \gamma\text{-}\mathrm{NZD}_{\overline{R}^{(r)}}\overline{\mathrm{Syz}_1^R M}^{(r)}=\overline{R}^{(r)}_1\setminus\{0\}.
$$ 
Hence  $\gamma\text{-}\mathrm{depth}_R\, \mathrm{Syz}_1^RM \geq r+1$.
\end{proof}

%
%
\section{Criteria for a $\gamma$-regular element}
\label{sec_r-reg-cr}

In this section, we provide criteria for a $\gamma$-regular element. 
To begin with, we recall the definition of $\mathfrak{m}\text{-full}$ property from \cite{watanabe1987m} 
and extend this notion for modules.

\begin{definition}\label{def_m-full} Let $V\subseteq U$ be graded modules and $ z\in R_1\setminus\{0\}$.\\
We say that $U$ is $\mathfrak{m}\textit{-full}$ in $V$ with respect to $z$ if the following hold:
$$
 \mathfrak{m}\,U\underset{V}:z=U.
 $$
\end{definition}

\begin{lemma}\label{lem_m-full}Assume that $z\in \mathrm{NZD}_R\,V$. Then the following are equivalent:
\begin{itemize}
\item[(1)]
$U$ is $\mathfrak{m}\textit{-full}$ in $V$ with respect to $z$.
\item[(2)]
$\varphi^{(z)}:V/U\overset{\times z}{\to}\left(V/\mathrm{m}U\right)(1)$ the multiplication map by $z$ is injective.
\item[(3)]
$0\underset{V/\mathfrak{m} U}:z=U/\mathfrak{m}U$.
\item[(4)]
The natural injective map 
$\iota: \mathrm{Tor}_1^R(U/\mathfrak{m}U, \overline{R})\overset{\sim}\to \mathrm{Tor}_1^R(V/\mathfrak{m}U, \overline{R})$ 
 is  an isomorphisim .
\end{itemize}
\end{lemma}

\begin{proof} The equivalence among (1), (2), and (3) follows from the following equations:
$$
\mathrm{Ker}\varphi^{(z)}=\left(\mathfrak{m}U\underset{V}:z\right)/U\text{ and }
0\underset{V/\mathfrak{m}U}:z=\left(\mathfrak{m}U\underset{V}:z\right)/\mathfrak{m}U.
$$
(4) is equivalent to (3) because the following hold:
$$
\mathrm{Tor}_1^R(V/\mathfrak{m}U, \overline{R})\simeq \left(0\underset{V/\mathfrak{m}U}:z\right)(-1)\text{ and }
\mathrm{Tor}_1^R(U/\mathfrak{m}U, \overline{R})\simeq \left(U/\mathfrak{m}U\right)(-1).
$$
\end{proof}

\begin{notation}\label{n-5} Let $ \pi:F_M\to M$ be  a minimal graded free cover of $ M$.
\begin{itemize}
\item[(1)]
$\mathrm{C}^R M=F_M/\mathfrak{m}\mathrm{Syz}_1^RM$.
\item[(2)]
$ \mathrm{C}_0^R M=M$ and $ \underbrace{\mathrm{C}^R(\cdots (\mathrm{C}^R}_{{i\text{-times}}} M))=\mathrm{C}_i^R M=F_M/\mathfrak{m}^{i}\mathrm{Syz}_1^RM$ for $ i\geq1$.
\end{itemize}
\end{notation}

\begin{remark}\label{rem_Ci} 
$ \mathrm{C}_i^R M\otimes\overline R \simeq F_M/\left(\mathfrak{m}^i\mathrm{Syz}_1^RM+zF_M\right)
\simeq
\mathrm{C}_i^{\overline{R}}\, \overline{M}$ for  $ i\geq1$, 
where $ z\in R_1\setminus\{0\}$, $\overline{R}=R/zR$ and $\overline{M}=M\otimes \overline{R}$. 
\end{remark}

\begin{lemma}\label{lem_r-reg-cr} 
Let $ \pi:F_M\to M$ be  a minimal graded free cover of $ M$. 
We denote $\overline{R}=R/zR$ and $\overline{M}=M\otimes \overline{R}$ for $ z\in R_1\setminus\{0\}$. 
The following hold:
\begin{enumerate}[label=(\arabic*)]
\item
$\beta_1^RM \leq \alpha(M;z)+\beta_1^{\overline{R}}\,\overline{M}$.
\item
The following are equivalent:
\begin{enumerate}[label=(\roman*)]
\item
$z\in\gamma\text{-}\mathrm{NZD}_RM$.
\item
$\beta_1^RM=\alpha(M;z)+\beta_1^{\overline{R}}\,\overline{M}$.
\item
$ \beta_{1}^R(M)=\alpha\left(\mathrm{C}^R M;z\right)$.
\item
$\mathrm{Syz}_1^RM$ is $\mathfrak{m}\textit{-full}$ in $F_M$ with respect to $z$. 
\end{enumerate}
\end{enumerate}
\end{lemma}

Before proving Lemma \ref{lem_r-reg-cr} above, we make the following observation:

\begin{observation}\label{obs_r-reg-cr} 
Let $F_M\overset{\pi_M}{\to} M$ be a minimal graded free cover of $M$. We denote $U=\mathrm{Syz}_1^RM$, 
$\overline{R}=R/zR$ and $\overline{X}=X\otimes \overline{R}$ for a graded moudle $X$.\\
There is the following exact sequence:
 {\footnotesize
 $$
0\to \mathrm{Tor}_1^R(U/\mathfrak{m}U, \overline{R})\overset{\iota}\to \mathrm{Tor}_1^R(F_{M}/\mathfrak{m}U, \overline{R})\to \mathrm{Tor}_1^R(M, \overline{R})\overset{\rho}\to \mathrm{Tor}_1^R(M, k)\to \mathrm{Tor}_1^{\overline{R}}(\overline{M}, k)\to0,
 $$}
which is induced as follows:\\
Here, we denote $\mathrm{T}_{\overline{R}}(M)=\mathrm{Tor}_1^R(M, \overline{R})$ and 
$\mathrm{T}_{k}(M)=\mathrm{Tor}_1^R(M, k)$.
Applying functors $ \_\otimes k$ and  $\_ \otimes\overline{R}$ on the following exact sequence:
 $$ 
 0\to U/\mathfrak{m}U\to F_M/\mathfrak{m}U \overset{\pi'}\to M \to 0,
 $$
we obtain the following commutative diagram, where the first row is exact:
 {\footnotesize
 $$
\begin{tikzcd}[row sep=0.5cm,column sep=0.3cm]
0\arrow[r]&\mathrm{T}_{\overline{R}}(U/\mathfrak{m}U)\arrow[r,"\iota"]&\mathrm{T}_{\overline{R}}(F_M/\mathfrak{m}U)\arrow[r]&\mathrm{T}_{\overline{R}}(M) \arrow[d,"\rho"]\arrow[r] & \overline{U/\mathfrak{m}U}\arrow[d, "\theta_2"]\arrow[r]   & \overline{F_M/\mathfrak{m}U} \arrow[r," \pi' \otimes \overline{R}"] & \overline{M}
\\
&&&\mathrm{T}_{k}(M)\arrow[r,"\theta_1" ']& U/\mathfrak{m}U \otimes k.& & \hspace{50pt}
\end{tikzcd}
$$}
 In the above diagram, we note the following:\\
(i) $ \theta_1$ and $ \theta_2$ are isomorphisms.\\
(ii) $ \mathrm{Ker}\left(\pi' \otimes\overline{R}\right)\simeq \dfrac{U+zF_M}{\mathfrak{m}U+zF_M}\simeq \mathrm{Tor}_1^{\overline{R}}(\overline{M}, k)$.
\end{observation}

Using the above Observation \ref{obs_r-reg-cr}, we prove Lemma  \ref{lem_r-reg-cr}.
 
\begin{proof}[Proof of Lemma \ref{lem_r-reg-cr}] 
(1) From the exact sequence of Observation \ref{obs_r-reg-cr}, we obtain the following inequality:
$$
\dim_K\mathrm{Tor}_1^R(M, k)\leq 
\dim_K\mathrm{Tor}_1^R(M, \overline{R})+\dim_K\mathrm{Tor}_1^{\overline{R}}(\overline{M}, k).
$$
This implies (1) since $\alpha(M;z)=\dim_K\mathrm{Tor}_1^R(M, \overline{R})$. \\
(2)  (i)$\Leftrightarrow$(ii): By definition, $z\in\gamma\text{-}\mathrm{NZD}_RM$ if and only if the map $\rho$ in the exact sequence of Observation \ref{obs_r-reg-cr} is injective. This is equivalent to the following equation holding:
$$
\beta_1^RM=\alpha(M;z)+\beta_1^{\overline{R}}\,\overline{M}.
$$
(ii)$\Leftrightarrow$(iii): In the exact sequence of Observation \ref{obs_r-reg-cr}, the map $\rho$ is injective if and only if the map $\iota$ is an isomorphism.This is equivalent to the following equation holding:
$$
\dim_K\mathrm{Tor}_1^R(U/\mathfrak{m}U,  \overline{R})=\beta_1^RM=\alpha\left(\mathrm{C}^R M;z\right)=\dim_K \mathrm{Tor}_1^R(F_{M}/\mathfrak{m}U, \overline{R}).
$$
(iii)$\Leftrightarrow$(iv): This follows from Lemma \ref{lem_m-full} .
\qed
\end{proof}

\begin{remark}\label{rem_r-reg-cr}
Assume that $R$ is a standard graded Noetherian commutative $K$-algebra. 
If $z\in \mathrm{NZD}R$, 
then statements (1) and (2) in Lemma \ref{lem_r-reg-cr} remain valid.
\end{remark}

%
%
\section{Criteria for a $\gamma$-regular sequence}
\label{sec_r-reg-cr}

In this section, we provide criteria for a $\gamma$-regular sequence. 
We denote $ \overline{R}^{(0)}=R$, $ \overline{M}^{(0)}=M$, 
$ \overline{R}^{(i)}=R/(z_1,\ldots,z_i)R$ and $ \overline{M}^{(i)}=M/(z_1,\ldots,z_i)M$, 
where $ z_1,\ldots,z_i\in R_1\setminus \{0\}$ for $ i\geq1$.

\begin{lemma}\label{lem_r-seq-cr} 
Let $z_1,\ldots,z_r$ be a sequence in $ R_1\setminus \{0\}$. The following hold:
\begin{enumerate}[label=(\arabic*)]
\item
$\beta_{1}^{\overline{R}^{(i-1)}}\left(\overline{M}^{(i-1)}\right)\leq\alpha\left(\overline{M}^{(i-1)};z_i\right)+\beta_{1}^{\overline{R}^{(i)}}\left(\overline{M}^{(i)}\right)$ for $ i=1,\ldots,r$.
\item
$\beta_{1}^R(M)\leq \sum\limits_{i=1}^{r}\alpha\left(\overline{M}^{(i-1)};z_i\right)+\beta_{1}^{\overline{R}^{(r)}}\left(\overline{M}^{(r)}\right)$.
\item
The following conditions are equivalent:
\begin{enumerate}[label=(\roman*)]
\item
$ z_1,\ldots,z_r$ is a $ M\text{-}\gamma\text{-sequence}$ .
\item
$\beta_{1}^R(M)\geq \sum\limits_{i=1}^{r}\alpha\left(\overline{M}^{(i-1)};z_i\right)+\beta_{1}^{\overline{R}^{(r)}}\left(\overline{M}^{(r)}\right)$.
\item
$\beta_{1}^R(M)=\sum\limits_{i=1}^{r}\alpha\left(\overline{M}^{(i-1)};z_i\right)+\beta_{1}^{\overline{R}^{(r)}}\left(\overline{M}^{(r)}\right)$.
\item
$\beta_{1}^{\overline{R}^{(i-1)}}\left(\overline{M}^{(i-1)}\right)
=\alpha\left(\overline{M}^{(i-1)};z_i\right)+\beta_{1}^{\overline{R}^{(i)}}\left(\overline{M}^{(i)}\right)$ for $ i=1,\ldots,r$.
\item
$\beta_{1}^{\overline{R}^{(i-1)}}\left(\overline{M}^{(i-1)}\right)
=\alpha\left( \mathrm{C}^{\overline{R}^{(i-1)}}\,\overline{M}^{(i-1)};z_i\right)$ for $ i=1,\ldots,r$. 
\end{enumerate}
\end{enumerate}
\end{lemma}

\begin{proof} (1) follows from Lemma \ref{lem_r-reg-cr} (1).\\
(2) is obtained by  summing up the inequalities in (1) from $ i=1$ to $r$.\\
(3) The equivalence among (i), (iv) and (v) follows directly from $\ref{lem_r-reg-cr}$ (2). \\
The equivalence between  (ii) and (iii) follows from (2).
If (iv) holds, then summing up the equations in (iv) from $ i=1$ to $r$, we obtain (iii).\\
We only need to prove that  (iii) implies (iv). If (iii) holds, then we have
 \begin{align*}
 \beta_{1}^R(M)
 &=
 \sum\limits_{i=1}^{r}\left(\beta_{1}^{\overline{R}^{(i-1)}}\left(\overline{M}^{(i-1)}\right)-\beta_{1}^{\overline{R}^{(i)}}\left(\overline{M}^{(i)}\right)\right)+\beta_{1}^{\overline{R}^{(r)}}\left(\overline{M}^{(r)}\right)\\
 &=\sum\limits_{i=1}^{r}\alpha\left(\overline{M}^{(i-1)};z_i\right)+\beta_{1}^{\overline{R}^{(r)}}\left(\overline{M}^{(r)}\right).
 \end{align*}
From the above equation, we obtain the following:
$$ \sum\limits_{i=1}^{r}\left(\alpha\left(\overline{M}^{(i-1)};z_i\right)+\beta_{1}^{\overline{R}^{(i)}}\left(\overline{M}^{(i)}\right)-\beta_{1}^{\overline{R}^{(i-1)}}\left(\overline{M}^{(i-1)}\right)\right)=0.$$
Hence, we conclude that the equations in (iv) hold, since by (1), 
$$ \alpha\left(\overline{M}^{(i-1)};z_i\right)+\beta_{1}^{\overline{R}^{(i)}}\left(\overline{M}^{(i)}\right)-\beta_{1}^{\overline{R}^{(i-1)}}\left(\overline{M}^{(i-1)}\right)\geq0\text{ for } i=1,\ldots,r. \qquad$$
\end{proof}

Using above Lemma \ref{lem_r-reg-cr} (3), we obtain the following result.

\begin{lemma}\label{lem_r-reg-cr-pred} 
Let $ z_1,\ldots,z_r$ be a $ M\text{-}\gamma\text{-sequence}$. 
If $  \mathrm{pd}_{\overline{R}^{(r)}}\overline{M}^{(r)}=0$, then the following hold:
\begin{enumerate}[label=(\arabic*)]
\item
$ \alpha\left( \mathrm{C}^{\overline{R}^{(j-1)}}\,\overline{M}^{(j-1)};z_j\right) 
=\sum\limits_{i=j}^{r}\alpha\left(\overline{M}^{(i-1)};z_i\right)$ for $ j=1,\ldots,r$.
\item
$ \sum\limits_{j=1}^r\alpha\left( \mathrm{C}^{\overline{R}^{(j-1)}}\,\overline{M}^{(j-1)};z_j\right)
 =n\beta_1^R(M)-\beta_2^R(M)$.
\end{enumerate}
\end{lemma}

\begin{proof} (1) We note that Lemma \ref{lem_r-reg-cr} (3) (v).
Summing up the equations in Lemma \ref{lem_r-reg-cr} (3) (iv) from $ i=j$ to $r$, we obtain 
$$
\alpha\left( \mathrm{C}^{\overline{R}^{(j-1)}}\,\overline{M}^{(j-1)};z_j\right)
=\beta_{1}^{\overline{R}^{(j-1)}}\left(\overline{M}^{(j-1)}\right)
=\sum\limits_{i=j}^{r}\alpha\left(\overline{M}^{(i-1)};z_i\right),
$$
since $\beta_{1}^{\overline{R}^{(r)}}\left(\overline{M}^{(r)}\right)=0$ by the assumption.\\
(2) Using (1), we have
$$
 \sum\limits_{j=1}^r\alpha\left( \mathrm{C}^{\overline{R}^{(j-1)}}\,\overline{M}^{(j-1)};z_j\right)
 =\sum\limits_{j=1}^r\sum\limits_{i=j}^{r}\alpha\left(\overline{M}^{(i-1)};z_i\right)
 = \sum\limits_{j=1}^rj\alpha\left( \overline{M}^{(j-1)};z_j\right).
$$
On the other hand, from Corollary \ref{cor_r-seq-pred} (2), we have 
 \begin{align*}
n\beta_1^R(M)-\beta_2^R(M)
&=n\sum\limits_{j=1}^r\alpha\left(\overline{M}^{(j-1)};z_j\right)-\sum\limits_{j=1}^r(n-j)\alpha\left(\overline{M}^{(j-1)};z_j\right)\\
&= \sum\limits_{j=1}^rj\alpha\left( \overline{M}^{(j-1)};z_j\right).
\end{align*}
Hence, the assertion holds.
\end{proof}

\begin{remark}\label{rem_r-reg-cr-pred}
Assume that $R$ is a standard graded Noetherian commutative Koszul $K$-algebra with $\dim_K\mathfrak{m}\otimes k =n$.
If we  impose the additional condition that  $z_1,\ldots,z_r$ form an $R$-regular sequence, 
then the assertions in Lemma \ref{lem_r-seq-cr} and Lemma \ref{lem_r-reg-cr-pred} remain valid 
by Remark \ref{rem_kos-reg} and Remark \ref{rem_r-reg-cr}.
\end{remark}

%
%
\section{$\hat\gamma$-regular elements and $\hat\gamma$-regular sequences}
\label{sec_rhat-elem-seq}

In this section, we introduce $\hat\gamma$-regular elements and $\hat\gamma$-regular sequences for a finitely generated graded module $M$. Under a suitable condition, a $\gamma$-regular sequence becomes a $\hat\gamma$-regular sequence.

\begin{definition and notation}\label{def-n_rhat-elem}
An element $z$ in $\displaystyle\bigcap_{i\geq0}\gamma\text{-}\mathrm{NZD}_R \mathrm{C}_i^R(M)$ is called an $M$-$\hat\gamma$-$regular$ element or $\hat\gamma$-$regular$ for M. We denote
$$
\hat\gamma\text{-}\mathrm{NZD}_R M=\displaystyle\bigcap_{i\geq0}\gamma\text{-}\mathrm{NZD}_R \mathrm{C}_i^R(M)
\text{ the set of }M\text{-}\hat\gamma\text{-regular elements}. 
$$
\end{definition and notation}

\begin{definition}\label{def_str-m-full}Let $V\subseteq U$ be graded modules and $ z\in R_1\setminus\{0\}$.\\
We say that $U$ is $\textit{strongly }\mathfrak{m}\textit{-full}$ in $V$ with respect to $z$ if the following hold:
$$
\mathfrak{m}^{i+1}\,U\underset{V}:z=\mathfrak{m}^{i}\,U\text{ for all }i\geq0,\text{ where }\mathfrak{m}^{0}=R.
$$
\end{definition}

\begin{remark}\label{rem_rhat-elem} The following hold:
\begin{enumerate}[label=(\arabic*)]
\item
$ \hat\gamma\text{-}\mathrm{NZD}_R M\subseteq\gamma\text{-}\mathrm{NZD}_R M$.
\item
$ z\in\hat\gamma\text{-}\mathrm{NZD}_R M$ if and only if $ \mathrm{Syz}_1^RM$ is strongly $\mathfrak{m}\text{-full}$ in $ F_M$ with respect to $z$, where $ \pi:F_M\to M$ is a minimal graded free cover of $ M$.
\end{enumerate}
\end{remark}

\begin{definition and notation}\label{def-n_rhat-seq}
Let $ \underline z=z_1,\ldots,z_r\in R_1\setminus\{0\}$  be a sequence of nonzero linear elements.
We denote
$$\overline R^{(i)}=R/(z_1,\ldots,z_i)R,\;\overline M^{(i)}=M\otimes\overline R^{(i)}\text{ for }i=1,\ldots,r.$$
We say that $\underline z$ is an $M\text{-}\gamma\textit{-regular sequence}$ of length $r$ if the following conditions hold: 
$$
\overline{z_i}\in \hat\gamma\text{-}\mathrm{NZD}_{\overline R^{(i-1)}}\overline M^{(i)},
$$
where $\overline{z_i}$ is a image of $z_i$ in $\overline R^{(i-1)}$ for $i=1,\ldots,r$.
We denote
$$
\hat\gamma\text{-}\mathrm{depth}_R M\text{ the maximal length of }M\text{-}\hat\gamma\text{-regular sequence}.
$$
\end{definition and notation}

\begin{lemma}\label{lem_rhat-seq-pred}
Let $ z_1,\ldots,z_r$ be an $M$-$\gamma$-regular sequence. 
If $  \mathrm{pd}_{\overline{R}^{(r)}}\overline{M}^{(r)}=0$, then the following hold:
\begin{enumerate}[label=(\arabic*)]
\item
$ \mathrm{pd}_{\overline{R}^{(r)}}\overline{\mathrm{C}_i^RM}^{(r)}
=\mathrm{pd}_{\overline{R}^{(r)}}\mathrm{C}_i^{\overline{R}^{(r)}}\overline{M}^{(r)}
=0$ for all $i\geq0$.
\item
$ z_1,\ldots, z_r$ is a $\mathrm{C}^RM$-$\gamma$-regular sequence.
\end{enumerate}
\end{lemma}

\begin{proof}
(1): Let $ F_M\to M$ be  a minimal graded free cover of $ M$. Then we note  the following:
$$
\overline{\mathrm{C}_i^RM}^{(r)} \simeq 
\mathrm{C}_i^{\overline{R}^{(r)}}\overline{M}^{(r)} 
\simeq F_M/\left(\mathfrak{m}^i\mathrm{Syz}_1^RM+(z_1,\ldots,z_r)F_M\right)
\text{ for  } i\geq0.
$$
Since $  \mathrm{pd}_{\overline{R}^{(r)}}\overline{M}^{(r)}=0$, 
we have $ \mathfrak{m}^i\mathrm{Syz}_1^RM\subseteq \mathrm{Syz}_1^RM\subseteq (z_1,\ldots,z_r)F_M$.
Hence 
$  \mathrm{pd}_{\overline{R}^{(r)}}\overline{\mathrm{C}_i^RM}^{(r)}
=\mathrm{pd}_{\overline{R}^{(r)}}\mathrm{C}_i^{\overline{R}^{(r)}}\overline{M}^{(r)}=0$
 for all $ i\geq0$.\\
(2) Since $\beta_{1}^{\overline{R}^{(r)}}(\overline{\mathrm{C}^RM}^{(r)})=0$, using Lemma \ref{lem_r-seq-cr} (3) (ii), it is enough to show that the following inequality holds:
$$
\beta_{1}^R(\mathrm{C}^RM)\geq 
\sum\limits_{j=1}^{r}\alpha\left(\overline{\mathrm{C}^RM}^{(j-1)};z_j\right)
=\sum\limits_{j=1}^r\alpha\left( \mathrm{C}^{\overline{R}^{(j-1)}}\,\overline{M}^{(j-1)};z_i\right).
$$
Since $ \mathrm{C}^{R}M\simeq F_M/\mathfrak{m}\mathrm{Syz}_1^{R}M$ 
and $ \mathrm{Tor}_1^R(M, k)\simeq \mathrm{Syz}_1^{R}M/\mathfrak{m}\mathrm{Syz}_1^{R}M$, 
 we have the following exact sequence:
$$ 
0\to \mathrm{Tor}_1^R(M, k) \to \mathrm{C}^{R}M \to M \to 0 .
$$
Applying the functor $ \_\otimes k$ on the above exact sequence,  we obtain the following exact sequence:
 $$
 \mathrm{Tor}_2^R(M, k) \to \mathrm{Tor}_1^R(k,\mathrm{Tor}_1^R(M, k)) \to \mathrm{Tor}_1^R(\mathrm{C}^{R}M, k) \to 0 .
 $$
Using Lemma \ref{lem_r-reg-cr-pred} (2) and the above exact sequence, we obtain the following inequality:
$$
\beta_1^R\left(\mathrm{C}^{R}M\right) \geq n\beta_1^R(M)-\beta_2^R(M) 
= \sum\limits_{i=1}^r\alpha\left( \mathrm{C}^{\overline{R}^{(j-1)}}\,\overline{M}^{(j-1)};z_i\right),
$$
since  $\mathrm{Tor}_1^R(k,\mathrm{Tor}_1^R(k,M))\simeq \mathrm{Tor}_1^R(k,k)^{\oplus \beta_1^R(M)}$ up to shifting.
\end{proof}

\begin{proposition}\label{prop_rhat-seq-pred}
If $ \underline{z}=z_1,\ldots,z_r$ is a $ M\text{-}\gamma\text{-sequence}$ 
with $  \mathrm{pd}_{\overline{R}^{(r)}}\overline{M}^{(r)}=0$, 
then $\underline{z}$ is also a $ M\text{-}\hat\gamma\text{-sequence}$.
\end{proposition}

\begin{proof}
Since $ \mathrm{C}_{i+1}^RM=\mathrm{C}^R\left(\mathrm{C}_i^RM\right)$ 
and $  \mathrm{pd}_{\overline{R}^{(r)}}\overline{\mathrm{C}_i^RM}^{(r)}=0$  
by Lemma $\ref{lem_rhat-seq-pred}$ (1), 
applying Lemma $\ref{lem_rhat-seq-pred}$ (2) repeatedly,  
we see that $\underline{z}$ is a $ \mathrm{C}_i^RM\text{-}\gamma\text{-sequence}$  
for all $ i\in\mathbb{Z}_{\geq0}$.

\end{proof}

\begin{remark}\label{rem_rhat-seq-pred}
Assume that $R$ is a standard graded Noetherian commutative Koszul $K$-algebra.
If we  impose the additional condition that  $z_1,\ldots,z_r$ form an $R$-regular sequence, 
then the assertions in Lemma \ref{lem_rhat-seq-pred} and Proposition \ref{prop_rhat-seq-pred} 
remain valid 
by Remark \ref{rem_r-reg-cr-pred}.
\end{remark}

%
%
\section{Modules with a first syzygy generated in a single degree}
\label{sec_single-gen}

In this section, we partially address the main theorem by presenting Proposition \ref{prop_sdg} for modules whose first syzygy is generated in a single degree.

\begin{remark}\label{rem_nzd} The following hold:
\begin{enumerate}[label=(\arabic*)]
\item
$\mathrm{NZD}_RM_{\geq j} \subseteq \mathrm{NZD}_RM/\Gamma_{\mathfrak{m}}(M)$  for any $j\in \mathbb{Z}$.\\
In fact, $\dim_K \left(0\underset{M}:z\right)<\infty$ for $z\in\mathrm{NZD}_RM_{\geq j}$ since $0\underset{M}:z \subseteq M_{\leq j-1}$.
\item
Let $r\in\mathbb{Z}$. The following are equivalent:
\begin{enumerate}[label=(\roman*)]
\item
$\mathrm{NZD}_RM_{\geq r}\ne \emptyset$.
\item
$M_{\geq r}\cap\Gamma_{\mathfrak{m}}(M)=0$.
\item
$\Gamma_{\mathfrak{m}}(M)_{\geq r}=0$.
\item
$\mathrm{deg}\Gamma_{\mathfrak{m}}(M)=\mathrm{deg}\,\mathrm{soc}M\leq r-1$.
\end{enumerate}
\item
If $\mathrm{NZD}_RM_{\geq r}\ne \emptyset$ for some $r\in\mathbb{Z}$, then 
$\mathrm{NZD}_RM_{\geq j}=\mathrm{NZD}_RM/\Gamma_{\mathfrak{m}}(M)$ for all $j\geq r$.
In fact, from (1), the following hold: 
$$
\mathrm{NZD}_RM_{\geq r}\subseteq
\mathrm{NZD}_RM_{\geq j} \subseteq
\mathrm{NZD}_RM/\Gamma_{\mathfrak{m}}(M)
\text{ for all } j\geq r.
 $$
 On the other hand, $\mathrm{NZD}_RM/\Gamma_{\mathfrak{m}}(M)\subseteq \mathrm{NZD}_RM_{\geq r}$
 since we have
 $$
 M_{\geq r}\simeq
 \left(M/\Gamma_{\mathfrak{m}}(M)\right)_{\geq r}\subseteq
 M/\Gamma_{\mathfrak{m}}(M)
$$ 
from the following exact sequence: 
$$
0\to\Gamma_{\mathfrak{m}}(M)_{\geq r}\to M_{\geq r} \to  \left(M/\Gamma_{\mathfrak{m}}(M)\right)_{\geq r} \to 0,
$$
where $\Gamma_{\mathfrak{m}}(M)_{\geq r}=0$ by (2).
\end{enumerate}
\end{remark}

\begin{lemma}\label{lem_nzd} 
Let $\mathrm{deg}\,\mathrm{Syz}_1^RM\otimes k =d$ and $\mathrm{indeg}\,\mathrm{Syz}_1^RM\otimes k =d_1$.
The following hold:
$$
\mathrm{NZD}_RM_{\geq d_1}\subseteq
\gamma\text{-}\mathrm{NZD}_RM\subseteq
\mathrm{NZD}_RM_{\geq d}.
$$
\end{lemma}

Before proving Lemma \ref{lem_nzd} above, we make the following observation:

\begin{observation}\label{obs_nzd} 
Let Let $\mathrm{deg}\,\mathrm{Syz}_1^RM\otimes k =d$, 
$\mathrm{indeg}\,\mathrm{Syz}_1^RM\otimes k =d_1$ and 
$\pi:F_M\to M$ be a minimal graded free cover of $ M$. 
We denote $U=\mathrm{Syz}_1^RM$ and $X=0\underset{M}:z$ for $ z\in R_1\setminus\{0\}$. 
There are  following exact sequences:
\begin{align}
(\text{I}) \quad0\to& U/\mathfrak{m}U \overset{\iota}
\to \left( \mathfrak{m}U\underset{F_M}:z\right)/\mathfrak{m}U \to X_{\geq d_1}\notag\\
(\text{II}) \quad0\to& \left(U/\mathfrak{m}U\right)_{\geq d} 
\overset{\iota_{\geq d}}\to \left(\left( \mathfrak{m}U\underset{F_M}:z\right)/\mathfrak{m}U\right)_{\geq d}
\to X_{\geq d} \to 0,\notag
\end{align}
which are induced as follows:\\
Applying snake lemma to the following commutative diagram with exact rows:
$$
\begin{tikzcd}[row sep=0.6cm,column sep=0.6cm]
0\arrow[r]
&U/\mathfrak{m}U \arrow[d,"\times z"]\arrow[r] 
&F_M/\mathfrak{m}U \arrow[d,"\times z"]\arrow[r]
& M \arrow[d, "\times z"]\arrow[r]
& 0\;\\
0\arrow[r]
&\left(U/\mathfrak{m}U\right)(1) \arrow[r] 
&\left(F_M/\mathfrak{m}U\right)(1) \arrow[r]
& M(1)\arrow[r] 
& 0,
\end{tikzcd}
$$
we have the following exact sequence for each $j\in\mathbb{Z}$:
$$
0\to \left(U/\mathfrak{m}U\right)_{\geq j} 
\to \left(\left( \mathfrak{m}U\underset{F_M}:z\right)/\mathfrak{m}U\right)_{\geq j}
\to X_{\geq j} \to \left(U/\mathfrak{m}U\right)_{\geq j+1} .
$$
If $j=d$, then $\left(U/\mathfrak{m}U\right)_{\geq d+1}=0$. Hence we obtain (II).\\
Here we remark that 
$\left( \mathfrak{m}U\underset{F_M}:z\right)_{\geq j}=
\left( \mathfrak{m}U\right)_{\geq j+1}\underset{F_M}:z$ for $j\in\mathbb{Z}$, 
since $z$ is a $F_M$-regular element.
If $j=d_1$, then the following hold:
$$
\left(U/\mathfrak{m}U\right)_{\geq d_1}\simeq
U_{\geq d_1}/\left(\mathfrak{m}U\right)_{\geq d_1}=
U/\mathfrak{m}U.
$$
$$
\left(\left( \mathfrak{m}U\underset{F_M}:z\right)/\mathfrak{m}U\right)_{\geq d_1}\simeq
\left( \left(\mathfrak{m}U\right)_{\geq d_1+1}\underset{F_M}:z\right)/\left(\mathfrak{m}U\right)_{\geq d_1}=
\left( \mathfrak{m}U\underset{F_M}:z\right)/\mathfrak{m}U.
$$
Hence we obtain (I).

\end{observation}

Using the above Observation \ref{obs_nzd}, we prove Lemma  \ref{lem_nzd}.

\begin{proof}[Proof of Lemma \ref{lem_nzd}] 
We note that $z\in \mathrm{NZD}_RM_{\geq j}$ for $j\in\mathbb{Z}$ if and only if 
$X_{\geq j}=\left(0\underset{M}:z\right)_{\geq j}=0$. 
If $z\in \mathrm{NZD}_RM_{\geq d_1}$, 
then $\iota$ in the exact sequence of Observation \ref{obs_nzd} (I) is an isomorphism,
 since $X_{\geq d_1}=0$.
Using Lemma \ref{lem_m-full} and Lemma \ref{lem_r-reg-cr} (2), we conclude that $z\in \gamma\text{-}\mathrm{NZD}_RM$. 
If $z\in \gamma\text{-}\mathrm{NZD}_RM$, then again using Lemma \ref{lem_m-full} and Lemma \ref{lem_r-reg-cr} (2), 
$\iota_{\geq d}$ in the exact sequence of Observation \ref{obs_nzd} (II) is an isomorphism.
Hence, $X_{\geq d}=0$. Therefore, we see that $z\in \mathrm{NZD}_RM_{\geq d}$.
 \qed
\end{proof}

As a corollary of the preceding Lemma \ref{lem_nzd} and Remark \ref{rem_nzd}, we obtain the following results.

\begin{corollary}\label{cor_nzd} 
If $\mathbf{deg}\left(\mathrm{Syz}_1^RM\otimes k\right) = \{d\}$,  then the following hold:
\begin{enumerate}[label=(\arabic*)]
\item
$\gamma\text{-}\mathrm{NZD}_RM=\mathrm{NZD}_RM_{\geq d}$.
\item
The following are equivalent:
\begin{enumerate}[label=(\roman*)]
\item
$\gamma\text{-}\mathrm{NZD}_RM\neq \emptyset$.
\item
$\mathrm{deg}\,\mathrm{soc}M\leq d-1$.
\end{enumerate}
\item
If $\gamma\text{-}\mathrm{NZD}_RM\neq \emptyset$, then 
$\gamma\text{-}\mathrm{NZD}_RM=\mathrm{NZD}_RM/\Gamma_{\mathfrak{m}}(M)$.
\end{enumerate}
\end{corollary}

\begin{remark}\label{rem_zar-open}
Let $M$ and $M'$ be two finitely generated graded $R$-modules. 
Assume that 
$\mathbf{deg}\left(\mathrm{Syz}_1^R M\otimes k\right) = \{d\}$ 
and 
$\mathbf{deg}\left(\mathrm{Syz}_1^R M'\otimes k\right) = \{d'\}$ 
for some $d, d' \in \mathbb{Z}$. Then the following hold:
\begin{enumerate}[label=(\arabic*)]
\item
$\gamma\text{-}\mathrm{NZD}_R M = \mathrm{NZD}_R\,M_{\geq d} \subseteq R_1$
is a Zariski open set, 
since $\mathfrak{p} \cap R_1$ is a $K$-linear subspace 
for each associated prime $\mathfrak{p}$ of  $M_{\geq d}$ when $M_{\geq d} \neq 0$.
\item
If $\gamma\text{-}\mathrm{NZD}_R M, \gamma\text{-}\mathrm{NZD}_R M' \neq \emptyset$, 
then $\gamma\text{-}\mathrm{NZD}_R M \cap \gamma\text{-}\mathrm{NZD}_R M' \neq \emptyset$  
by (1),  since $K$ is an infinite field.
\end{enumerate}
\end{remark}

\begin{notation}\label{n-7}  Let $N$ be a graded submodule of $M$ and $ z\in R_1\setminus\{0\}$ be a nonzero linear element.
$$
 N\underset{M}:\mathfrak{m}^{\infty}=\left\{\xi\in M \mid  \mathfrak{m}^{i}\xi\subseteq N\right\}
 \text{  the satulation of }N\text{ in }M.
 $$
\end{notation}

\begin{remark}\label{rem_sat}
Let $\pi:F_M\to M$ be a minimal graded free cover of $ M$.\\
We denote  
$U=\mathrm{Syz}_1^RM$ and $\widetilde{U}=\mathfrak{m}^iU\underset{F_M}:\mathfrak{m}^{\infty}$ 
for $i\geq0$. The following hold:\\
$$
M/\Gamma_{\mathfrak{m}}(M)\simeq
M/\Gamma_{\mathfrak{m}}\left(\mathrm{C}_i^RM\right)\simeq
F_M/\widetilde{U}
\text{ for }i\geq0.
$$

\end{remark}

\begin{lemma}\label{lem_nzd-sg} 
If $\mathbf{deg}\left(\mathrm{Syz}_1^RM\otimes k\right) = \{d\}$,  then the following hold:
\begin{enumerate}[label=(\arabic*)]
\item
$M_{\geq d+i} \simeq \left(\mathrm{C}_i^RM\right)_{\geq d+i}$ for all $i\geq0$.
\item
$\gamma\text{-}\mathrm{NZD}_R\mathrm{C}_i^RM=\mathrm{NZD}_RM_{\geq d+i}$ for all $i\geq0$.
\item
$\gamma\text{-}\mathrm{NZD}_R\mathrm{C}_{i}^RM\subset
\gamma\text{-}\mathrm{NZD}_R\mathrm{C}_{i+1}^RM$ for all $i\geq0$.
\item
$\gamma\text{-}\mathrm{NZD}_RM=\hat\gamma\text{-}\mathrm{NZD}_RM$.
\item
If $\underline{z}=z_1,\ldots,z_r$ is a $M$-$\gamma$-regular sequence, then 
$\underline{z}$ is also a $\mathrm{C}_{i}^RM$-$\gamma$-regular sequence for all $i\geq0$. 
Especially, $\gamma\text{-}\mathrm{depth}_RM\leq\gamma\text{-}\mathrm{depth}_R\mathrm{C}^RM$.
\item
If $\gamma\text{-}\mathrm{NZD}_R\mathrm{C}_{i}^RM\neq \emptyset$, 
then $\gamma\text{-}\mathrm{NZD}_R\mathrm{C}_{i}^RM=\mathrm{NZD}_RM/\Gamma_{\mathfrak{m}}(M)$.
\item
$\gamma\text{-}\mathrm{NZD}_RM\neq \emptyset$ if and only if 
$\mathrm{Syz}_1^RM= \left(\mathrm{Syz}_1^RM\underset{F_M}:\mathfrak{m}^{\infty}\right)_{\geq d}$.
\item
Let $i\geq0$. The following are equivalent:
\begin{enumerate}[label=(\roman*)]
\item
$\gamma\text{-}\mathrm{NZD}_R\mathrm{C}_{i}^RM\neq \emptyset$.
\item
$i\geq \mathrm{deg}\,\mathrm{soc}M-d+1$.
\item
$\mathfrak{m}^i\mathrm{Syz}_1^RM= \left(\mathrm{Syz}_1^RM\underset{F_M}:\mathfrak{m}^{\infty}\right)_{\geq d+i}$.
\end{enumerate}
\end{enumerate}
\end{lemma}

\begin{proof} 
(1) It is enough to show that 
$M_{\geq d+1}\simeq\left(\mathrm{C}^RM\right)_{\geq d+1}$.\\
This is obtained from the following exact sequence:
$$
0\to \left(\mathrm{Syz}_1^RM/\mathfrak{m}\mathrm{Syz}_1^RM\right)_{\geq d+1}
\to \left(\mathrm{C}^RM\right)_{\geq d+1}
\to M_{\geq d+1} \to 0,
$$
where $\left(\mathrm{Syz}_1^RM/\mathfrak{m}\mathrm{Syz}_1^RM\right)_{\geq d+1}=0$ 
since 
$\mathbf{deg}\left(\mathrm{Syz}_1^RM\otimes k\right) = \{d\}$.\\
(2) follows from (1) and Corollary \ref{cor_nzd} (1), 
since $\mathbf{deg}\left(\mathrm{Syz}_1^R\mathrm{C}_i^RM\otimes k\right) = \{d+i\}$. \\
(3) follows from (2). (4) follows from (3).\\
(5) follows from (4) since $\mathrm{C}_i^R M\otimes\overline{R}^{(j)}
\simeq
\mathrm{C}_i^{\overline{R}^{(j)}}\, \overline{M}^{(j)}$ by Remark \ref{rem_Ci}, 
where we denote
$ \overline{R}^{(j)}=R/(z_1,\ldots,z_j)R$ and $ \overline{M}^{(j)}=M/(z_1,\ldots,z_j)M$ for $j=1,\cdots,r$.\\
(6) follows from (2) and Remark \ref{rem_nzd} (3). \\
(7) Since $\left(\mathrm{Syz}_1^RM\right)_{\geq d}=\mathrm{Syz}_1^RM$, 
there is the following exact sequence:
$$
0\to\mathrm{Syz}_1^RM \hookrightarrow \left(\mathrm{Syz}_1^RM\underset{F_M}:\mathfrak{m}^{\infty}\right)_{\geq d} 
\to \Gamma_{\mathfrak{m}}(M)_{\geq d} \to 0.
$$
Using the exact sequence above, 
$\mathrm{Syz}_1^RM= \left(\mathrm{Syz}_1^RM\underset{F_M}:\mathfrak{m}^{\infty}\right)_{\geq d}$ 
if and only if 
$\Gamma_{\mathfrak{m}}(M)_{\geq d}=0$. This is equivalent to 
$\gamma\text{-}\mathrm{NZD}_RM\neq \emptyset$
from Corollary \ref{cor_nzd} (2).\\
(8) The equivalence between (i) and (ii) follows from Remark \ref{rem_nzd} (2),
since 
$$
\gamma\text{-}\mathrm{NZD}_R\mathrm{C}_i^RM=\mathrm{NZD}_RM_{\geq d+i}\text{ by (2)}.
$$
The equivalence between (i) and (iii) follows from (7),  since 
$$
\mathrm{Syz}_1^R\mathrm{C}_i^RM\underset{F_M}:\mathfrak{m}^{\infty}=
\mathfrak{m}^i\mathrm{Syz}_1^RM\underset{F_M}:\mathfrak{m}^{\infty}=
\mathrm{Syz}_1^RM\underset{F_M}:\mathfrak{m}^{\infty}.
$$
\end{proof}

\begin{notation}\label{n-8} 
$\delta_R(M)=\min\{i\geq0\mid \gamma\text{-}\mathrm{depth}_R\mathrm{C}_i^RM=n\}$.
\end{notation}

\begin{lemma}\label{lem_delta-sg} 
The following hold:
\begin{enumerate}[label=(\arabic*)]
\item
If $n=1$,  then $\delta_R(M)= 0$.
\item
If $\mathrm{pd}_RM= 0$,  then $\delta_R(M)=0$.
\item
If $\mathbf{deg}\left(\mathrm{Syz}_1^RM\otimes k\right) = \{d\}$,  then $\delta_R(M)<\infty$.
\end{enumerate}
\end{lemma}

\begin{proof}
(1) If $n=1$, i.e., $ R=K\lbrack x_1 \rbrack$, then $ x_1\in R_1$ is a $ \gamma\text{-regular}$ element 
for any nonzero finitely generated $R$ module $M$, 
since $M$ is a direct sum of principal modules of the form $ K\lbrack x_1 \rbrack/(x_1^r)$ for some positive integer $r$
  or $ K\lbrack x_1 \rbrack$, up to shifts.\\
(2) If $\mathrm{pd}_RM= 0$, then $\mathrm{depth}_RM=\gamma\text{-}\mathrm{depth}_RM=n$. 
Hence, $\delta_R(M)=0$.\\
(3) We prove this by induction on $n$, the number of variables of base ring $R$. If $n=1$, the assertion holds by (1).
Let $n>1$. We choose an element $z_1\in\mathrm{NZD}_RM/\Gamma_{\mathfrak{m}}(M)$ 
and denote  $\overline{R}=R/z_1R$ and $\overline{M}=M/z_1M$. 
Then, by the induction hypothesis, we have $\delta_{\overline{R}}(\overline{M})<\infty$. 
Hence, there exists a sequence $z_2, \cdots, z_n \in R_1\setminus\{0\}$ 
such that the images $\overline{z_2}, \cdots, \overline{z_n}$ in $\overline{R}$ 
form a $\mathrm{C}_{i}^{\overline{R}}\,\overline{M}$-$\gamma$-regular sequence 
for all $i\geq\delta_{\overline{R}}(\overline{M})$.
Moreover, we note that 
$$
\mathrm{C}_i^R M\otimes\overline R \simeq\mathrm{C}_{i}^{\overline{R}}\,\overline{M}
$$
by Remark \ref{rem_Ci}.
Here, we set 
$$
r=\max\{\delta_{\overline{R}}(\overline{M}), \mathrm{deg}\,\mathrm{soc}M-d+1\}.
$$
Using Lemma \ref{lem_nzd-sg} (6) and (8), we obtain  $\gamma\text{-}\mathrm{depth}_R\mathrm{C}_r^RM=n$,
 since the sequence $z_1,\cdots, z_n$ is a $\mathrm{C}_{r}^RM$-$\gamma$-regular sequence. 
 Thus, we conclude that $\delta_R(M)\leq r<\infty$.
\end{proof}

\begin{lemma}\label{lem_soc-sg} 
Assume that  $\mathbf{deg}\left(\mathrm{Syz}_1^RM\otimes k\right)=\{d\}$  for some $d\in\mathbb{Z}$.\\ 
If $ \mathrm{reg}_R\,\mathrm{Syz}_1^RM=d$ ,  then the following hold:
\begin{enumerate}[label=(\arabic*)]
\item
$\mathbf{deg}\,\mathrm{soc}M\subseteq\{d-1\}$.
\item
$ \gamma\text{-}\mathrm{NZD}_RM\neq\emptyset$.
\item
One of the following condition holds for $ z\in\gamma\text{-}\mathrm{NZD}_RM$:
\begin{enumerate}[label=(\roman*)]
\item
$ \mathrm{Syz}_1^{\overline{R}}\overline{M}=0$.
\item
$\mathbf{deg}\left(\mathrm{Syz}_1^{\overline{R}}\overline{M}\otimes k\right)=\{d\}$ and 
$\mathrm{reg}_{\overline{R}}\,\mathrm{Syz}_1^{\overline{R}}\overline{M}=d$.
\end{enumerate}
\end{enumerate}
\end{lemma}

\begin{proof}
(1) We recall that $ R=K\lbrack x_1,\ldots,x_n\rbrack$ is a polynomial ring in $n$ variables. 
The following hold:
$$
\mathrm{soc}M(-n) \simeq
\mathrm{H}_n(x_1,\ldots,x_n;M)\simeq
\mathrm{Tor}^R_{n}(M, k)\simeq
 \mathrm{Tor}^R_{n-1}(\mathrm{Syz}_1^RM, k),
 $$
where $\mathrm{H}_{n}(x_1,\ldots,x_{n};M)$ denotes the $n\text{-th}$ Koszul homology of $M$ 
with respect to  $x_1,\ldots,x_{n}$.
So, we obtain
 $$ 
 \mathbf{deg}\left(\mathrm{soc}M(-n)\right)=
 \mathbf{deg}\left(\mathrm{Tor}^R_{n-1}(\mathrm{Syz}_1^RM, k)\right)\subseteq\{d+n-1\},
 $$
since $\mathrm{Syz}_1^RM$ has a linear resolution. 
Therefore, $\mathbf{deg}\,\mathrm{soc}M\subseteq\{d-1\}$.\\
(2) This follows from (1) and Corollary $\ref{cor_nzd}$ (2).\\
(3): We recall that 
$\mathrm{Syz}_1^RM\otimes\overline{R}\simeq
\mathrm{Syz}_1^{\overline{R}}\overline{M}\oplus\mathrm{soc}M(-1)$ 
form Lemma $\ref{lemA_splitting}$ (1).\\
Hence, if $ \mathrm{Syz}_1^{\overline{R}}\overline{M}\neq 0$, then 
$ \mathbf{deg}\left(\mathrm{Syz}_1^{\overline{R}}\overline{M}\otimes k\right)=\{d\}$ 
and $ \mathrm{reg}_{\overline{R}}\,\mathrm{Syz}_1^{\overline{R}}\overline{M}\geq d$.\\
Moreover, from Lemma $\ref{lemA_splitting}$ (8),  we have
 $$
 d=\mathrm{reg}_R\left(\mathrm{Syz}_1^RM\right)=
 \mathrm{max}\left\{
 \mathrm{reg}_{\overline{R}}\left(\mathrm{Syz}_1^{\overline{R}}\overline{M}\right),\;
 \mathrm{sup}\,(\mathrm{Syz}_1^RM\otimes k)
 \right\}.
 $$
Therefore, we obtain $\mathrm{reg}_{\overline{R}}\,\mathrm{Syz}_1^{\overline{R}}\overline{M}= d$. 
\end{proof}

\begin{proposition}\label{prop_sdg}
Let  $\mathbf{deg}\left(\mathrm{Syz}_1^RM\otimes k\right)=\{d\}$. 
Then the following are equivalent:
\begin{enumerate}[label=(\arabic*)]
\item
$M$ has a componetwise linear first syzygy.
\item
$\gamma$-$\mathrm{depth}M=n$.
\end{enumerate}
\end{proposition}

\begin{proof}
(2) $ \Rightarrow$ (1): It is enough to show that 
$$
\mathrm{reg}_R\,\left(\mathrm{Syz}_1^RM\right)_{\langle d+i \rangle}=
\mathrm{reg}_R\,\mathfrak{m}^i\mathrm{Syz}_1^RM=d+i
$$
for all $i\geq 0$. However, since 
$n=\gamma\text{-}\mathrm{depth}_RM=
\gamma\text{-}\mathrm{depth}_R\mathrm{C}_i^RM$
 for all $i\geq 0$ from Lemma \ref{lem_nzd-sg} (5), we only need to show the case when $i=0$. 
Choosing an $ M\text{-}\gamma\text{-}$ sequence $ z_1,\ldots,z_n$, 
we then have $\mathrm{pd}_{\overline{R}^{(n)}} \overline{M}^{(n)}=0$, where
we denote $\overline{R}^{(n)}=R/(z_1,\ldots,z_n)R=k$ and $\overline{M}^{(n)}=M\otimes \overline{R}^{(n)}$.
Using Corollary \ref{cor_r-seq-pred} (1), we obtain 
$\mathrm{reg}_R\,\mathrm{Syz}_1^RM=\mathrm{deg}\,(\mathrm{Syz}_1^RM\otimes k)=d$.\\
(1) $ \Rightarrow$ (2): We prove this by induction on $n$ the number of variables. 
If $n=1$, then $\gamma\text{-}\mathrm{depth}_R M=1$ from Lemma \ref{lem_delta-sg} (1).
If $n>1$,  then ,using Lemma \ref{lem_soc-sg} (2), 
$\gamma\text{-}\mathrm{NZD}_RM\neq\emptyset$ and we can take an element $ z\in\gamma\text{-}\mathrm{NZD}_RM$.
Therefore, it is enough to show that 
$\gamma\text{-}\mathrm{depth}_{\overline{R}}\overline{ M}=n-1$.
\begin{enumerate}[label=(\roman*)]
\item
If $ \mathrm{Syz}_1^{\overline{R}}\overline{M}=0$,   
then we note that $ \overline{M}$ is a free $ \overline{R}\text{-moudule}$.  
Hence 
$\gamma\text{-}\mathrm{depth}_{\overline{R}}\overline{ M}=\mathrm{depth}_{\overline{R}}\overline{ M}=n-1$.
\item
If $ \mathrm{Syz}_1^{\overline{R}}\overline{M}\neq0$, 
then $ \mathbf{deg}\left(\mathrm{Syz}_1^{\overline{R}}\overline{M}\otimes k\right)=\{d\}$ 
and  $ \mathrm{reg}_{\overline{R}}\,\mathrm{Syz}_1^{\overline{R}}\overline{M}=d$
by Lemma  $\ref{lem_soc-sg}$ (3).
Hence, applying the induction hypothesis, 
we obtain $ \gamma\text{-}\mathrm{depth}_{\overline{R}}\overline{ M}=\text{Kull-dim}\overline{R}=\nu-1$.
\end{enumerate}
This completes the proof. 
\end{proof}

From the argument in the proof of Proposition $\ref{prop_sdg}$ above, 
we derive the following remark.

\begin{remark}\label{rem_seq-red}
Assume that $ \mathbf{deg}\left(\mathrm{Syz}_1^RM\otimes k\right)=\{d\}$.
If $ \gamma\text{-}\mathrm{depth}_R M=n$, 
then $\gamma\text{-}\mathrm{depth}_{\overline{R}} \overline{M}=n-1$ 
for any $ z\in\gamma\text{-}\mathrm{NZD}_RM$, 
where $ \overline{R}=R/zR$ and $ \overline{M}=M\otimes\overline{R}$, 
since $ \mathrm{Syz}_1^{\overline{R}}\overline{M}=0$ or 
$ \mathrm{reg}_{\overline{R}}\,\mathrm{Syz}_1^{\overline{R}}\overline{M}=d$.
\end{remark}

%
%
\section{$\hat\gamma\text{-sequence}$ revisited}
\label{sec_rhat_revisited}

We provide the proof of our main theorem, which we restate as Theorem \ref{thm_main} in this section.

\begin{notation}\label{n-9} 
Let $ F_M\to M$ be  a minimal free cover of $ M$.
$$
\mathrm{C}^R_{\langle j \rangle} M=F_M/\left(\mathrm{Syz}_1^RM\right)_{\langle j \rangle}\text{ for }  j\in\mathbb{Z}.
$$
\end{notation}

\begin{remark}\label{rem_C<j>} 
$\mathrm{C}_{\langle j \rangle}^R M\otimes\overline R \simeq F_M/\left(\left(\mathrm{Syz}_1^RM\right)_{\langle j \rangle}+zF_M\right)\simeq\mathrm{C}_{\langle j \rangle}^{\overline{R}} \overline{M}$ for $ j\in\mathbb{Z}$, 
where $ z\in R_1\setminus\{0\}$, $\overline{R}=R/zR$ and $\overline{M}=M\otimes \overline{R}$. 
\end{remark}

\begin{lemma}\label{lem_mfull-C}
Let $F_M\to M$ be  a minimal free cover of $ M$ 
and denote $U=\mathrm{Syz}_1^RM$.
Then the following hold:
\begin{enumerate}[label=(\arabic*)]
\item
For $i\geq0$, $z\in\gamma\text{-}\mathrm{NZD}_R\,\mathrm{C}^R_{i} M$ 
if and only if the following multiplication map by $z$  is injective:
$$
\varphi^{(z,i)}:F_M/\mathfrak{m}^iU \overset{\times z}\to \left(F_M/\mathfrak{m}^{i+1}U\right)(1).
$$
\item
For $j\in\mathbb{Z}$, $z\in\gamma\text{-}\mathrm{NZD}_R\,\mathrm{C}^R_{\langle j \rangle} M$ 
if and only if the following multiplication map by $z$  is injective:
$$
\psi^{(z,j)}:F_M/U_{\langle j \rangle} \overset{\times z}\to \left(F_M/\mathfrak{m}U_{\langle j \rangle}\right)(1).
$$
\end{enumerate}
\end{lemma}

\begin{proof} Both (1) and (2) follow from Lemma \ref{lem_m-full} (2) and Lemma \ref{lem_r-reg-cr}  (2).
\end{proof}

\begin{lemma}\label{lem_rhat-cw} 
Let $\mathbf{deg}(\mathrm{Syz}_1^R M \otimes k) = \{d_1 < \cdots < d_r\}$ for some $d_1, \cdots, d_r \in \mathbb{Z}$.
The following equations hold:
\begin{enumerate}[label=(\arabic*)]
\item
$\hat{\gamma}\text{-}\mathrm{NZD}_R(M) =
\displaystyle{ \bigcap_{j \in \mathbb{Z}} }\gamma\text{-}\mathrm{NZD}_R\,\mathrm{C}^R_{\langle j \rangle} M$.
\item
$\hat{\gamma}\text{-}\mathrm{NZD}_R(M) =
\displaystyle{ \bigcap_{i=1}^r }\gamma\text{-}\mathrm{NZD}_R\,\mathrm{C}^R_{\langle d_i \rangle} M$.
\end{enumerate}
\end{lemma}

\begin{proof} (1) Let $F_M\to M$ be  a minimal free cover of $ M$ 
and denote $U=\mathrm{Syz}_1^RM$. 
Recall that 
$
\hat\gamma\text{-}\mathrm{NZD}_R M=
\displaystyle\bigcap_{i\geq0}\gamma\text{-}\mathrm{NZD}_R \mathrm{C}_i^R(M)
$ 
by definition and note the follwoing:
\begin{itemize}
\item[(a)]
Since $ (\mathfrak{m}_1)^{i+r}U_{j-r}=
(\mathfrak{m}_1)^{i}(\mathfrak{m}_1)^{r}U_{j-r}\subseteq(\mathfrak{m}_1)^{i}U_{j}$
 for $ i,r\geq0$ and $j\in\mathbb{Z}$, we obtain the follwoing:
\begin{align*}
\left(\mathfrak{m}^iU\right)_{i+j} &=
(\mathfrak{m}_1)^{i}U_{j}+(\mathfrak{m}_1)^{i+1}U_{j-1}+\cdots+(\mathfrak{m}_1)^{i+r}U_{j-r}+\cdots \\
& =(\mathfrak{m}_1)^{i}U_{j}=\left(U_{\langle j \rangle}\right)_{i+j}
\end{align*}
for all $i\geq0$ and $ j\in\mathbb{Z}$.
\item[(b)]
Since $ z\in \mathrm{NZD}_R\,F_M$ and $ \left(U_{\langle j \rangle}\right)_{\leq j-1}=\left(\mathfrak{m}U_{\langle j \rangle}\right)_{\leq j}=0$ for $ j\in\mathbb{Z}$, we conclude that
$$
\left(\psi^{(z,j)}\right)_{\leq j-1}:\left(F_M/U_{\langle j \rangle}\right)_{\leq j-1} \overset{\times z}\longrightarrow \left(F_M/\mathfrak{m}U_{\langle j \rangle}\right)_{\leq j}
$$
is an injective $ K\text{-linear}$ map for all $ j\in\mathbb{Z}$.
\end{itemize}
Using (a), for each $i\in\mathbb{Z}_{\geq 0}$ and $ j\in\mathbb{Z}$, 
we have the following commutative diagram:
$$ 
\begin{matrix}
\left(\varphi^{(z,i)}\right)_{i+j}:&\left(V/\mathfrak{m}^iU\right)_{i+j} & \overset{\times z}\longrightarrow &\left(V/\mathfrak{m}^{i+1}U\right)_{i+j+1}& \\ 
&\mid\mid &  &\mid\mid& \\
&V_{i+j}/(\mathfrak{m}_1)^{i}U_j& \overset{\times z}\longrightarrow &V_{i+j+1}/(\mathfrak{m}_1)^{i+1}U_j& \\ &\mid\mid&  &\mid\mid&\\ 
\left(\psi^{(z,j)}\right)_{i+j}:&\left(V/U_{\langle j \rangle}\right)_{i+j}&\overset{\times z}\longrightarrow &\left(V/\mathfrak{m}U_{\langle j \rangle}\right)_{i+j+1}&.
\end{matrix}$$
Using the above commutative diagram and (b), we can see that the following are equivalent:
\begin{enumerate}[label=(\roman*)]
\item
$\mathrm{Ker}\, \varphi^{(z,i)}=0$ for all $ i\in\mathbb{Z}_{\geq 0}$.
\item
$\mathrm{Ker}\, \psi^{(z,j)}=0$ for all $ j\in\mathbb{Z}$.
\end{enumerate}
Hence, we conclude that  
$\hat{\gamma}\text{-}\mathrm{NZD}_R(M) =
\displaystyle{ \bigcap_{j \in \mathbb{Z}} }\gamma\text{-}\mathrm{NZD}_R\,\mathrm{C}^R_{\langle j \rangle} M$ 
by Lemma \ref{lem_mfull-C}.\\
(2)  By (1), it is enough to show that
 for each $j\in\mathbb{Z}$, there exists  $1 \leq i_j \leq r$ such that
 $$
\gamma\text{-}\mathrm{NZD}_R\,\mathrm{C}^R_{\langle j \rangle} M
\supseteq 
\gamma\text{-}\mathrm{NZD}_R\,\mathrm{C}^R_{\langle d_{j_i} \rangle} M.
 $$
If $\mathrm{pd}_R \mathrm{C}^R_{\langle j \rangle} M = 0$, then 
$\gamma\text{-}\mathrm{NZD}_R\,\mathrm{C}^R_{\langle j \rangle} M
= R_1 \setminus \{0\} \supseteq
\gamma\text{-}\mathrm{NZD}_R\,\mathrm{C}^R_{\langle d_1 \rangle} M$.\\
If $\mathrm{pd}_R \mathrm{C}^R_{\langle j \rangle} M \neq 0$, then 
there exists $1 \leq i_j \leq r$ such that 
$$
(\mathrm{Syz}_1^R M)_{\langle j \rangle} = \mathfrak{m}^{q} (\mathrm{Syz}_1^R M)_{\langle d_{i_j} \rangle}, \text{i.e., } \mathrm{C}^R_{\langle j \rangle} M = \mathrm{C}^R_q\left(\mathrm{C}^R_{\langle d_{i_j} \rangle} M\right),
$$
where $q=j-d_{i_j}$. 
Hence, using Lemma \ref{lem_nzd-sg} (3), we have
$$
\gamma\text{-}\mathrm{NZD}_R\,\mathrm{C}^R_{\langle j \rangle} M =
 \gamma\text{-}\mathrm{NZD}_R\left(\mathrm{C}^R_q\left(\mathrm{C}^R_{\langle d_{i_j} \rangle} M\right)\right) \supseteq 
 \gamma\text{-}\mathrm{NZD}_R\,\mathrm{C}^R_{\langle d_{i_j} \rangle} M.
 $$
\end{proof}

\begin{proposition}\label{prop_rhat-cw} 
Let $\mathbf{deg}(\mathrm{Syz}_1^R M \otimes k) = \{d_1 < \cdots < d_r\}$ for some $d_1, \cdots, d_r \in \mathbb{Z}$.
The following are equivalent:
\begin{enumerate}[label=(\arabic*)]
\item
$\hat{\gamma}\text{-}\mathrm{depth}_R M = n$.
\item
$\gamma\text{-}\mathrm{depth}_R \mathrm{C}^R_{\langle j \rangle} M = n$ for all $j \in \mathbb{Z}$.
\item
$\mathrm{reg}_R\left(\mathrm{Syz}_1^R M\right)_{\langle j \rangle}\leq j$ for all $j \in \mathbb{Z}$.
\item
$\gamma\text{-}\mathrm{depth}_R \mathrm{C}^R_{\langle d_i \rangle} M = n$ for $i = 1, \ldots, r$.
\end{enumerate}
\end{proposition}

\begin{proof}
(1) $\Rightarrow$ (2): 
Choose an $M\text{-}\hat{\gamma}\text{-}$sequence $z_1, \ldots, z_{n}$ of length $n$. 
By Lemma $\ref{lem_rhat-cw}$ (1), we have
 $$
 z_1 \in \hat{\gamma}\text{-}\mathrm{NZD}_R(M) = 
 \bigcap_{j \in \mathbb{Z}} \gamma\text{-}\mathrm{NZD}_R\,\mathrm{C}^R_{\langle j \rangle} M.
 $$ 
 Hence, 
 $z_1$ is a $\gamma\text{-regular}$ element on $\mathrm{C}^R_{\langle j \rangle} M$ 
 for all $j \in \mathbb{Z}$. 
We note that 
$$\mathrm{C}^R_{\langle j \rangle} M \otimes \overline{R}^{(1)}\simeq 
\mathrm{C}^{\overline{R}^{(1)}}_{\langle j \rangle} \overline{M}^{(1)}$$ 
by Remark $\ref{rem_C<j>}$, 
where $\overline{R}^{(1)} = R / z_1 R$ and $\overline{M}^{(1)} =M \otimes\overline{R}^{(1)}$. 
Again by Lemma $\ref{lem_rhat-cw}$ (1), we have
$$
\overline{z_2} \in \hat{\gamma}\text{-}\mathrm{NZD}_{\overline{R}^{(1)}}\,\overline{M}^{(1)} = \bigcap_{j \in \mathbb{Z}} \gamma\text{-}\mathrm{NZD}_{\overline{R}^{(1)}}\,\mathrm{C}^{\overline{R}^{(1)}}_{\langle j \rangle} \overline{M}^{(1)},
$$
where $\overline{z_2}$ is the image of $z_2$ in $\overline{R}^{(1)}$. \\
Hence, $\overline{z_2}$ is a $\gamma\text{-regular}$ element on 
$\mathrm{C}^R_{\langle j \rangle} M\otimes \overline{R}^{(1)} $ for all $j \in \mathbb{Z}$, and so on.
Thus, we conclude that the $M\text{-}\hat{\gamma}\text{-}$sequence $z_1, \ldots, z_{n}$ 
is a $\mathrm{C}^R_{\langle j \rangle} M\text{-}\gamma\text{-sequence}$ and 
$\gamma\text{-}\mathrm{depth} \mathrm{C}^R_{\langle j \rangle} M = n$ for all $j \in \mathbb{Z}$.\\
(2) $\Leftrightarrow$ (3): 
If $\mathrm{pd}_R\mathrm{C}^R_{\langle j \rangle} M=0$, i.e., 
$\left(\mathrm{Syz}_1^R M\right)_{\langle j \rangle}=0$, 
then both (2) and (3) hold, since 
$\gamma\text{-}\mathrm{depth}_R \mathrm{C}^R_{\langle j \rangle} M =\mathrm{depth}_R \mathrm{C}^R_{\langle j \rangle} M= n$ and 
$\mathrm{reg}_R\left(\mathrm{Syz}_1^R M\right)_{\langle j \rangle}=-\infty$.
If $\mathrm{pd}_R\mathrm{C}^R_{\langle j \rangle} M\neq0$, 
then the equivalence between (2) and (3) follows from Proposition $\ref{prop_sdg}$.\\
(2) $\Rightarrow$ (4): This is clear.\\
(4) $\Rightarrow$ (1): 
 We prove this by induction on $n$ the number of variables. 
If $n = 1$, then $\hat{\gamma}\text{-}\mathrm{depth}_R M ={\gamma}\text{-}\mathrm{depth}_R M= 1$  from Lemma \ref{lem_delta-sg} (1). 
If $n > 1$, then  we can take an element 
$$
z \in \hat{\gamma}\text{-}\mathrm{NZD}_R(M) = \bigcap_{i=1}^r \gamma\text{-}\mathrm{NZD}_R\left(\mathrm{C}^R_{\langle d_i \rangle} M\right) \neq \emptyset
$$ 
from Lemma $\ref{lem_rhat-cw}$ (2) and Remark $\ref{rem_zar-open}$. 
It is enough to show that 
$$
\hat{\gamma}\text{-}\mathrm{depth}_{\overline{R}} \overline{M} = n-1,
$$
where $\overline{R} = R / z R$ and $\overline{M} = \overline{R} \otimes M$. 
From Remark $\ref{rem_seq-red}$, we have
$$
\gamma\text{-}\mathrm{depth} \mathrm{C}^{\overline{R}}_{\langle d_i \rangle} \overline{M} =n - 1
$$
for $i=1,\cdots,r$.
Since $\mathbf{deg}(\mathrm{Syz}_1^{\overline{R}}\overline{M} \otimes k) \subseteq 
\{d_1 < \cdots < d_r\}$, 
applying the induction hypothesis, 
we obtain $\hat{\gamma}\text{-}\mathrm{depth}_{\overline{R}} \overline{M} = n- 1$. 
\end{proof}

\begin{theorem}\label{thm_main} 
Then the following are equivalent:
\begin{enumerate}[label=(\arabic*)]
\item
$M$ has a componetwise linear first syzygy.
\item
$\gamma$-$\mathrm{depth}M=n$.
\end{enumerate}
\end{theorem}

\begin{proof}
Using Proposition $\ref{prop_rhat-seq-pred}$, 
if we choose $\underline{z} = z_1, \ldots, z_{n}$ as an $M\text{-}\gamma\text{-}$sequence, 
then $\underline{z}$ is also an $M\text{-}\hat{\gamma}\text{-}$sequence, since 
$
\mathrm{pd}_{\overline{R}^{(n)}} \overline{M}^{(n)} = \mathrm{pd}_k \overline{M}^{(n)} = 0,
$
where $\overline{R}^{(n)} = R / (z_1, \ldots, z_{n})R$ and $\overline{M}^{(n)} =M\otimes  \overline{R}^{(n)} $. 
Hence, equation (2) is equivalent to $\hat{\gamma}\text{-}\mathrm{depth}_R M = n$. 
Therefore, the assertion follows from Proposition $\ref{prop_rhat-cw}$. 
\end{proof}

\begin{proposition}\label{prop_delta} 
$\delta_R(M)<\infty$.
\end{proposition} 

\begin{proof} If $\mathrm{pd}_RM=0$, then $\delta_R(M)=0$ by Lemma  \ref{lem_delta-sg} (2).
We assume that $\mathbf{deg}(\mathrm{Syz}_1^R M\otimes k) = \{d_1 < \cdots < d_r\}$ for some $d_1, \cdots, d_r \in \mathbb{Z}$. From Lemma \ref{lem_delta-sg} (3), we set
$$
m=\max\left\{\delta_R\left(\mathrm{C}^R_{\langle d_i \rangle}M\right)\mid i=1,\cdots,r\right\}<\infty.
$$
Using (a) in the proof of Proposition \ref{prop_rhat-cw},  we obtain
$$
\left(\mathfrak{m}^mU\right)_{\langle d_i +m\rangle}=
\mathfrak{m}^mU_{\langle d_i \rangle}
\text{, i.e., }\mathrm{C}^R_{\langle d_i +m\rangle}\left(\mathrm{C}^R_mM\right)=
\mathrm{C}^R_m\left(\mathrm{C}^R_{\langle d_i\rangle}M\right)
$$
for $i=1,\cdots,r$, where we denote $U=\mathrm{Syz}_1^R M$. Hence, we have
$$
\gamma\text{-}\mathrm{depth}_R\mathrm{C}^R_{\langle d_i +m\rangle}\left(\mathrm{C}^R_mM\right)=
\gamma\text{-}\mathrm{depth}_R\mathrm{C}^R_m\left(\mathrm{C}^R_{\langle d_i\rangle}M\right)=n
$$
for $i=1,\cdots,r$. Since we have
$$
\mathbf{deg}\left(\mathrm{Syz}_1^R \mathrm{C}^R_mM\otimes k  \right) =
\mathbf{deg}\left(\mathfrak{m}^m\mathrm{Syz}_1^RM\otimes k  \right)
\subseteq \{d_1+m < \cdots < d_r+m\},
$$
we obtain 
$\hat\gamma\text{-}\mathrm{depth}_R\mathrm{C}^R_mM=
\gamma\text{-}\mathrm{depth}_R\mathrm{C}^R_mM=n
$ 
by Proposition  \ref{prop_rhat-cw}. Thus, we conclude that $\delta_R(M)\leq m<\infty$.
\end{proof}

\begin{notation}\label{n-10}
$\mathrm{cd}_R\,M = 
\min\{i \geq 0 \mid \mathrm{Syz}_{i+1}^RM \text{ is a componentwise linear module}\}.$
\end{notation}

\begin{remark}\label{rem_cd}
We note that
$$
\mathrm{pd}_R\,M = 
\min\{i \geq 0 \mid \mathrm{Syz}_{i+1}^RM =0 \}\quad \text{for } M \neq 0
$$
and $0$ is a componentwise linear module by definition. The following inequalities hold:
$$
\mathrm{cd}_R\,M \leq \mathrm{pd}_R\,M \leq n \quad \text{for } M \neq 0.
$$
\end{remark}

\begin{proposition}\label{prop_eb} 
$\mathrm{cd}_R\,M + \gamma\text{-}\mathrm{depth}_R\,M \leq n$.
\end{proposition}

\begin{proof}
Set $c=n- \gamma\text{-}\mathrm{depth}_R\,M$. It is enough to show that 
$\mathrm{Syz}_{c+1}^RM$ is a componentwise linear module. 
Using Lemma \ref{lem_r-seq-syz} (2), we obtain
$$
\gamma\text{-}\mathrm{depth}_R\,\mathrm{Syz}_{c}^RM=n.
$$
Hence, $\mathrm{Syz}_{c+1}^RM$ is a componentwise linear module by Theorem \ref{thm_main}.
\end{proof}

%
%
\section{Finite dimensional modules with linear syzygies}
\label{sec_fdim}

We begin this short section with the following simple examples.

\begin{example}\label{ex_fdim} 
Let $M = R / \mathfrak{m}^{r+1}$ for some $r\geq 0$. 
We denote $R^{[i]} = K[x_1, \ldots, x_i]$, 
$M^{[i]} = R^{[i]} / (x_1, \ldots, x_i)^{r+1}$ for $i = 1, \ldots, n$ 
and $R^{[0]} = M^{[0]} = K$. 
Then the following hold:
\begin{enumerate}[label=(\arabic*)]
\item 
From Corollary \ref{cor_nzd}  (1), 
the sequence of variables $x_1, \ldots, x_i$ is a $\gamma\text{-regular}$ sequence 
on $M^{[i]}$ as a module over $R^{[i]}$ for $i = 1, \ldots, n$.
 \item 
From Theorem \ref{thm_main}, $M^{[i]}$ has has a componetwise linear first syzygy as a module over $R^{[i]}$ 
, since $\gamma\text{-}\mathrm{depth}_{R^{[i]}}M^{[i]}=i$  for $i = 1, \ldots, n$ by (1).
 \item 
 Using Lemma \ref{lem_r-seq} (4), we obtain the following:
$$
\mathrm{p}_{R^{[i]}}\left(M^{[i]}\right) = 
\mathrm{p}_{R^{[i-1]}}\left(M^{[i-1]}\right) + 
\begin{pmatrix} i+r-1 \\ i-1 \end{pmatrix} (1+t)^{i-1}t
$$
for $i = 1, \ldots, n$ and $\mathrm{p}_{R^{[0]}}\left(M^{[0]}\right) = 1$. 
Hence, we have
\begin{align*}
\mathrm{p}_{R}\left(M\right) &= 
1 + \sum\limits_{i=1}^{\nu} \begin{pmatrix} i+r-1 \\ i-1 \end{pmatrix} (1+t)^{i-1} t\\
\mathrm{p}_{R}\left(\mathfrak{m}^{r+1}\right) &= 
\sum\limits_{i=1}^{\nu} \begin{pmatrix} i+r-1 \\ i-1 \end{pmatrix} (1+t)^{i-1}.
\end{align*}
 \item 
 For example, if $n = 3$ and $r = 2$, then
$$
 \mathrm{p}_{R}\left(\mathfrak{m}^{3}\right)  
 = \begin{pmatrix} 2 \\ 0 \end{pmatrix} + 
     \begin{pmatrix} 3 \\ 1 \end{pmatrix} (1+t) + 
     \begin{pmatrix} 4 \\ 2 \end{pmatrix} (1+t)^2 
= 10 + 15t + 6t^2.
$$
Since $\mathfrak{m}^{3}$ has a linear resolution by (2), we obtain
$$
    \mathrm{P}_{R}\left(\mathfrak{m}^{3}\right)(t, u) = u^3 (10 + 15tu + 6t^2u^2)
$$
the graded Poincare series of $\mathfrak{m}^3$. 
\end{enumerate}
\end{example}

\begin{remark}\label{rm_ord}
Unfortunately, a $\gamma\text{-regular}$ sequence depends on its order.
 Let $R = K[x, y]$ and $\mathfrak{m} = (x, y)$, i.e., $n = 2$. 
 We set $I = \mathfrak{m}x + \mathfrak{m}^3$ and $M = R / I$. The sequence $\underline{a} = y, x$ is an $M\text{-}\gamma\text{-regular}$ sequence, but the order-reversed sequence $\underline{a'} = x, y$ is not an $M\text{-}\gamma\text{-regular}$ sequence. 
 On the other hand, let $u = x+y$ and $v = x-y$. Then, both $\underline{b} = u, v$ and $\underline{b'} = v, u$ are $M\text{-}\gamma\text{-regular}$ sequences, provided that $\mathrm{char}(K) \neq 2$.
 \end{remark}

 \begin{lemma}\label{lem_fdim}
Let $\mathbf{deg}(\mathrm{Syz}_1^R M\otimes k) = \{d\}$ for some $d \in \mathbb{Z}$. 
If $\dim_K M < \infty$, then the following are equivalent:
\begin{enumerate}[label=(\arabic*)]
\item 
$\gamma\text{-}\mathrm{NZD}_R M \neq \emptyset$.
\item 
$M$ is a direct sum of principal modules of the form 
$R / \mathfrak{m}^n$ for some positive integer $n$, up to shifts.
\item 
$\gamma\text{-}\mathrm{depth}_R M =n$.
\item 
$M$ has a componetwise linear first syzygy.
\end{enumerate}
\end{lemma}

\begin{proof}
(3) $\Leftrightarrow$ (4): 
This follows from Theorem \ref{thm_main}.\\  
(2) $\Rightarrow$ (3): 
This follows from Example \label{ex_fdim}. \\ 
(3) $\Rightarrow$ (1): 
This holds from definitions. \\ 
(1) $\Rightarrow$ (2): 
Let $F_M = \bigoplus_{j \in \mathbb{Z}} R^{m_j}(-j) \to M$ be a minimal graded free cover of $M$. 
We note that the follwing holds:
$$
\mathrm{Syz}_1^RM\underset{F_M}:\mathfrak{m}^{\infty} = F_M. 
$$
By applying Lemma \ref{lem_nzd-sg} (7), we obtain
$$
\mathrm{Syz}_1^R M = (F_M)_{\geq d} = \bigoplus_{j \in \mathbb{Z}} \left(R^{m_j}(-j)\right)_{\geq d} = \bigoplus_{j \in \mathbb{Z}} \left(\mathfrak{m}^{d-j}\right)^{m_j} = \bigoplus_{j > d} \left(\mathfrak{m}^{d-j}\right)^{m_j},
$$
where we denote $\mathfrak{m}^{d-j} = R$ if $d-j \leq 0$. 
However, note that $m_j = 0$ when $d-j \leq 0$, as $F_M$ is a minimal graded free cover of $M$. 
Thus, $M$ can be expressed as a direct sum of principal modules of the form $R / \mathfrak{m}^n$ 
for some positive integer $n$. This completes the proof. 
\end{proof}

%
%
\section{Modules with linear syzygies  over a polynomial ring in two variables}
\label{sec_2var}
In this section, we assume that $R = K[x_1, x_2]$, i.e., $n = 2$. 
We denote $\overline{R} = R/Rz$ and $\overline{M} = M\otimes \overline{R}$ 
for $z \in R_1 \setminus \{0\}$ and  $M$ a finitely generated graded module over $R$.

\begin{lemma}\label{lem_2-dim}
The following are equivalent:
\begin{enumerate}[label=(\arabic*)]
\item 
$\gamma\text{-}\mathrm{NZD}_R(M) \neq \emptyset$.
\item 
$\beta_1^R(M) = \alpha(M; z) + \beta_1^{\overline{R}}(\overline{M})$ for some $z \in R_1 \setminus \{0\}$.
\item 
$M$ has a componetwise linear first syzygy, i.e., $\gamma\text{-}\mathrm{depth}_R(M) = 2$.
\end{enumerate}
\end{lemma}

\begin{proof}
(1) $\Leftrightarrow$ (2): This follows from Lemma \ref{lem_r-reg-cr} (2).  \\
(3) $\Rightarrow$ (1): This holds from definitions.\\ 
(1) $\Rightarrow$ (3): 
Let $z \in \gamma\text{-}\mathrm{NZD}_R(M)$, 
then $\gamma\text{-}\mathrm{NZD}_{\overline{R}}(\overline{M}) \neq \emptyset$ 
by Lemma \ref{lem_delta-sg}  (1) since $\text{Kull-dim}(\overline{R}) = 1$.
Hence, $\gamma\text{-}\mathrm{depth}_R(M) = 2$. 
\end{proof}

\begin{remark}\label{rm_2-dim}
The following hold:
\begin{enumerate}[label=(\arabic*)]
\item  
$\gamma\text{-}\mathrm{NZD}_R(M) = \hat{\gamma}\text{-}\mathrm{NZD}_R(M)$ 
since any $M\text{-}\gamma\text{-regular}$ element can be extended to an 
$M\text{-}\gamma\text{-regular}$ sequence of length 2.
\item
If $\mathrm{pd}_R(M) \leq1$, then 
$\mathrm{NZD}_R M = \gamma\text{-}\mathrm{NZD}_R M \neq \emptyset$ 
by Remark \ref{rem_r-seq} (1), since $\mathrm{depth}_R M\geq1$ 
for the sake of the Auslander-Buchsbaum formula. 
Especially, Using Lemma \ref{lem_2-dim}, $M$ has a componetwise linear first syzygy.
\end{enumerate}
\end{remark}

We write $f_1 \mid f_2$ to denote that $f_2$ is divisible by $f_1$ for $f_1, f_2 \in R$.

\begin{lemma}\label{lem_2-dim-prin}
Let $0 \neq M \simeq R/I$ be a nonzero principal module with 
$\mathrm{pd}_R(M) =2$.
We set
$\mathrm{indeg}(I\otimes k) = d$, 
$J_1 = I_{\langle d \rangle} \underset{R}{:} \mathfrak{m}^\infty$, 
$J = I \underset{R}{:} \mathfrak{m}^\infty$, 
and $\mathrm{indeg}(J\otimes k)= e$. 
Then the following hold:
\begin{enumerate}[label=(\arabic*)]
\item
$J = Rf$ and $J_1 = Rf'$ for some homogeneous polynomials $f, f' \in R$ with $f \mid f'$.
\item 
$\mathrm{NZD}_R\left(\mathrm{C}^R_{\langle d \rangle}M / \Gamma_{\mathfrak{m}}(\mathrm{C}^R_{\langle d \rangle}M)\right) \subseteq \mathrm{NZD}_R\left(M / \Gamma_{\mathfrak{m}}M\right)$.
\item 
$I_{\langle d \rangle} \not\subseteq Rz$ 
for any $z \in \mathrm{NZD}_R\left(\mathrm{C}^R_{\langle d \rangle}M / \Gamma_{\mathfrak{m}}(\mathrm{C}^R_{\langle d \rangle}M)\right)$.
\item 
$\dim_K \frac{J}{I + zJ} = d - e$ 
for any $z \in \mathrm{NZD}_R\left(\mathrm{C}^R_{\langle d \rangle}M / \Gamma_{\mathfrak{m}}(\mathrm{C}^R_{\langle d \rangle}M)\right)$.
\item 
$\alpha(M; z) = d - e$ 
for any $z \in \mathrm{NZD}_R\left(\mathrm{C}^R_{\langle d \rangle}M / \Gamma_{\mathfrak{m}}(\mathrm{C}^R_{\langle d \rangle}M)\right)$.
\end{enumerate}
\end{lemma}

\begin{proof}
(1) We note that $\mathrm{pd}_R(M / \Gamma_{\mathfrak{m}}(M)) \leq 1$ since $\mathrm{depth}_R(M / \Gamma_{\mathfrak{m}}(M)) \geq 1$ by the Auslander-Buchsbaum formula. Hence $J$ is a free module of rank 1, i.e., $J = Rf$ for some homogeneous element $f \in R$. \\
Similarly, $\mathrm{pd}_R\left(\mathrm{C}^R_{\langle d \rangle}M / \Gamma_{\mathfrak{m}}(\mathrm{C}^R_{\langle d \rangle}M)\right) \leq 1$, so $J_1 = Rf'$ for some homogeneous element $f' \in R$. 
Since $Rf = J \supseteq J_1 = Rf'$, we obtain $f \mid f'$.\\
(2) From (1), we obtain
$$
\mathrm{C}^R_{\langle d \rangle}M / \Gamma_{\mathfrak{m}}(\mathrm{C}^R_{\langle d \rangle}M) \simeq 
R / J_1 \simeq 
R / f'R\text{ and }M / \Gamma_{\mathfrak{m}}M \simeq 
R / J \simeq R / fR.
$$
Since $f \mid f'$, the assertion follows.\\
(3) If $\mathrm{C}^R_{\langle d \rangle}M / \Gamma_{\mathfrak{m}}(\mathrm{C}^R_{\langle d \rangle}M) = 0$, then $\dim_K(R / I_{\langle d \rangle}) < \infty$. 
Hence $I_{\langle d \rangle} \not\subseteq Rz$ for any $z \in \mathrm{NZD}_R\left(\mathrm{C}^R_{\langle d \rangle}M / \Gamma_{\mathfrak{m}}(\mathrm{C}^R_{\langle d \rangle}M)\right) = \mathrm{NZD}_R(0) = R_1 \setminus \{0\}$. \\
Assume that $\mathrm{C}^R_{\langle d \rangle}M / \Gamma_{\mathfrak{m}}(\mathrm{C}^R_{\langle d \rangle}M) \neq 0$. Let $z \in \mathrm{NZD}_R\left(\mathrm{C}^R_{\langle d \rangle}M / \Gamma_{\mathfrak{m}}(\mathrm{C}^R_{\langle d \rangle}M)\right)$.\\
If $I_{\langle d \rangle} \subseteq Rz$, then $Rf' = J_1 = I_{\langle d \rangle} \underset{R}{:} \mathfrak{m}^\infty \subseteq Rz \underset{R}{:} \mathfrak{m}^\infty = Rz$. \\
Hence, 
$z \notin \mathrm{NZD}_R\left(\mathrm{C}^R_{\langle d \rangle}M / \Gamma_{\mathfrak{m}}(\mathrm{C}^R_{\langle d \rangle}M)\right)=\mathrm{NZD}_R\left(R/Rf'\right)$.
This is a contradiction. Hence $I_{\langle d \rangle} \not\subseteq Rz$.\\
(4) We note that our assumption $\mathrm{pd}_R(M) =2$ implies $d>e$, since $\mathrm{depth}_RM=0$ by the Auslander-Buchsbaum formula.
Since $I \subseteq J = Rf$,  we have
$$
I = I'f
$$ for some homogeneous ideal $I' \subseteq R$. 
Here, we note that $\mathrm{indeg}(I'\otimes k)+e = d$. 
Using (3), $I_{\langle d \rangle} = I'_{\langle d - e \rangle}f \not\subseteq zR$. 
Hence, $I'_{\langle d - e \rangle} \not\subseteq zR$. 
Thus, we conclude that 
$$
I' + Rz = Rw^{d - e} + Rz
$$ 
for some $w \in R_1 \setminus \{0\}$. 
We obtain the following:
$$
 \frac{J}{I + zJ} = \frac{Rf}{I'f + Rfz} \simeq \frac{R}{I' + Rz} \simeq K[w] / (w^{d - e}).
$$
This implies
$$\dim_K \frac{J}{I + zJ} = d - e.$$
(5) Since $z \in \mathrm{NZD}_R\left(\mathrm{C}^R_{\langle d \rangle}M / \Gamma_{\mathfrak{m}}(\mathrm{C}^R_{\langle d \rangle}M)\right) \subseteq \mathrm{NZD}_R\left(M / \Gamma_{\mathfrak{m}}\,M\right)$ by (2), we have
$$
\alpha(M; z) = \dim_K\left(0 \underset{M}{:} z\right) 
=\dim_K\left(0 \underset{\Gamma_{\mathfrak{m}}\,M}{:} z\right) 
= \dim_K\left(0 \underset{J/I}{:} z\right).
$$
Here, we note that $\dim_K(J / I) < \infty$.  We obtain
$$
 \alpha(M; z) = \dim_K\mathrm{Coker}\left(J/I \overset{\times z}{\longrightarrow} J/I\right)
 = \dim_K\frac{J}{I + zJ} = d-e,
$$
where the last equality follows from (4). 
\end{proof}

We have slightly extended Theorem 4 in \cite{watanabe1987m} as follows:

\begin{theorem}\label{thm_2-dim-prin}
Let $0 \neq M \simeq R/I$ be a nonzero principal module with $\mathrm{pd}_R(M) =2$. 
Then the following are equivalent:
\begin{enumerate}[label=(\arabic*)]
\item 
$\beta_1^R(M) = \mathrm{indeg}(I\otimes k) - \mathrm{indeg} \left(\Gamma_{\mathfrak{m}}\,M\right) + 1$.
\item 
$M$ has a componetwise linear first syzygy.
\end{enumerate}
\end{theorem}

\begin{proof}
(1) $\Rightarrow$ (2): 
Let $\overline{R} = R / Rz$ for 
$z \in \mathrm{NZD}_R\left(\mathrm{C}^R_{\langle d \rangle}M / \Gamma_{\mathfrak{m}}(\mathrm{C}^R_{\langle d \rangle}M)\right)$ 
and $J = I \underset{R}{:} \mathfrak{m}^\infty$.
We note that 
$\mathrm{indeg}(J\otimes k)= \mathrm{indeg} \left(\Gamma_{\mathfrak{m}}\,M\right)$ 
and $\beta_1^{\overline{R}}(\overline{M}) = 1$. 
Using Lemma \ref{lem_2-dim-prin} (5), we obtain
$$
\beta_1^R(M) = 
\mathrm{indeg}(I\otimes k)- \mathrm{indeg} \left(\Gamma_{\mathfrak{m}}\,M\right) + 1 = 
\alpha(M; z) + \beta_1^{\overline{R}}(\overline{M}).
$$
Hence (2) holds from Lemma $\ref{lem_2-dim}$.  \\
(2) $\Rightarrow$ (1): By Lemma \ref{lem_2-dim} and Remark \ref{rm_2-dim} (1),  
we can choose an element
$$
z \in \gamma\text{-}\mathrm{NZD}_R(M) = \hat{\gamma}\text{-}\mathrm{NZD}_R(M) \subseteq \gamma\text{-}\mathrm{NZD}_R\left(\mathrm{C}^R_{\langle d \rangle}M\right). $$
Then, by Lemma \ref{lem_r-reg-cr} (2), the following holds:
$$
\beta_1^R(M) = \alpha(M; z) + \beta_1^{\overline{R}}(\overline{M}).
$$
Since
$
\gamma\text{-}\mathrm{NZD}_R\left(\mathrm{C}^R_{\langle d \rangle}M\right) =
 \mathrm{NZD}_R\left(\mathrm{C}^R_{\langle d \rangle}M / \Gamma_{\mathfrak{m}}(\mathrm{C}^R_{\langle d \rangle}M)\right)
$ by Corollary \ref{cor_nzd} (3), 
using Lemma \ref{lem_2-dim-prin} (5), we obtain
$$
\beta_1^R(M) =   \mathrm{indeg}(I\otimes k) - \mathrm{indeg} \left(\Gamma_{\mathfrak{m}}\,M\right) + 1.
$$
\end{proof}

\begin{lemma}\label{lem_2-dim-full}
Let $0 \neq M \simeq R/I$ be a nonzero principal module.
The following hold:
\begin{enumerate}[label=(\arabic*)]
\item 
Let $z\in  \mathrm{NZD}_R\left(M / \Gamma_{\mathfrak{m}}M\right)$, 
$\overline{R}=R/zR$ 
and $\overline{M}=M\otimes\overline{R}$.\\
If $\mathrm{pd}_R(M) \neq 0$, then $\beta_1^{\overline{R}}(\overline{M})=1$.
\item 
Let $z \in R_1 \setminus \{0\}$. The following are equivalent:
\begin{enumerate}[label=(\roman*)]
\item $z\in  \gamma\text{-}\mathrm{NZD}_R(M)$.
\item $0\underset{M}:z=\mathrm{soc}M$.
\end{enumerate}
\end{enumerate}
\end{lemma}

\begin{proof}
(1) It is enough to show that $I\not\subseteq Rz$.
Using the same argument as in the proof of Lemma \ref{lem_2-dim-prin} (1),
we can see that $I\underset{R}:\mathfrak{m}^{\infty}=Rf$ for some 
homogeneous polynomial $f \in R$ and  $M / \Gamma_{\mathfrak{m}}M\simeq R/Rf$.
If $I\subseteq Rz$, then $Rf \subseteq Rz$ ,i.e., 
$z\not\in  \mathrm{NZD}_R\left(M / \Gamma_{\mathfrak{m}}M\right)$. This is a contradiction.
Hence, $I\not\subseteq Rz$.\\
(2) If $\mathrm{pd}_R(M)= 0$, then the assertion clearly holds.\\
Assume that $\mathrm{pd}_R(M) \neq 0$. Then we note that $M\otimes Q(R)=0$, 
where $Q(R)=K(x_1, x_2)$ the quotient field of $R$,  
and the following holds:
$$
\beta_1^R(M) =\beta_0^R(M)+\beta_2^R(M)=1+\dim_K\mathrm{soc}M.
$$
If (ii) holds, then 
$z\in  \mathrm{NZD}_R\left(M / \Gamma_{\mathfrak{m}}M\right)$.
Using (1), we obtain
$$
\beta_1^R(M) =\beta_1^{\overline{R}}(\overline{M})+\alpha(M; z).
$$
Hence, $z\in  \gamma\text{-}\mathrm{NZD}_R(M)$ from Lemma \ref{lem_r-reg-cr} (2).
The converse implication is clear.
\end{proof}

\begin{lemma}\label{lem_2-dim-prin-sdg}
Let $0 \neq M \simeq R/I$ be a nonzero principal module.
Assume that $\mathbf{deg}(I \otimes k) = \{d\}$ and $\mathrm{pd}_R(M) =2$.
Then the following are equivalent:
\begin{enumerate}[label=(\arabic*)]
\item 
$M$ has a componetwise linear first syzygy.
 \item 
 $I = \mathfrak{m}^{d-e}f$ for some homogeneous element $f \in R$ with $\mathrm{indeg}(Rf) = e$.
\end{enumerate}
\end{lemma}

\begin{proof}
(1) $\Rightarrow$ (2): From Lemma \ref{lem_2-dim-prin} (1), 
$I \underset{R}{:} \mathfrak{m}^\infty = Rf$ 
for some homogeneous element $f \in R$.
Since $\gamma\text{-}\mathrm{NZD}_R(M) \neq \emptyset$, 
applying Lemma $\ref{lem_nzd-sg}$ (7), we obtain $I = (Rf)_{\geq d} = \mathfrak{m}^{d-e}f$, 
where $\mathrm{indeg}(Rf) = e$.\\  
(2) $\Rightarrow$ (1): From Lemma $\ref{lem_nzd-sg}$ (7), 
$\gamma\text{-}\mathrm{NZD}_R(M) \neq \emptyset$. 
Hence, $M$ has a componetwise linear first syzygy by Lemma \ref{lem_2-dim}.
\end{proof}

\begin{theorem}\label{thm_2-dim-prin-mdg}
Let $0 \neq M \simeq R/I$ be a nonzero principal module. 
Assume that $\mathbf{deg}(I\otimes k) = \{d_1 < \cdots < d_r\}$ and $\mathrm{pd}_R(M) =2$. 
Then the following are equivalent:
\begin{enumerate}[label=(\arabic*)]
\item
$M$ has a componetwise linear first syzygy.
\item
There exist homogeneous polynomials $f_i \in R$ with 
$\mathrm{indeg}(Rf_i) = e_i$ for $i = 1, \ldots, r$ such that
$f_{i+1} \mid f_i$ for $i = 1, \ldots, r-1$ and
$$
I = \mathfrak{m}^{d_1 - e_1}f_1 + \cdots + \mathfrak{m}^{d_r - e_r}f_r.
$$
\end{enumerate}
\end{theorem}

\begin{proof}
(1) $\Rightarrow$ (2): 
We note that 
$I = I_{\langle d_1 \rangle} + \cdots + I_{\langle d_r \rangle}$. 
Using Lemma \ref{lem_2-dim-prin-sdg}, 
we have 
$$I_{\langle d_i \rangle} \underset{R}{:} \mathfrak{m}^\infty= Rf_i
\text{ and }
I_{\langle d_i \rangle} = \mathfrak{m}^{d_i - e_i}f_i
$$ 
for some homogeneous element $f_i \in R$ with 
$\mathrm{indeg}(Rf_i) = e_i$ for $i = 1, \ldots, r$. 
Moreover, $f_{i+1} \mid f_i$,
since $I_{\langle d_{i+1} \rangle} \supseteq \mathfrak{m}^{d_{i+1} - d_i}I_{\langle d_i \rangle}$
and 
$$
I_{\langle d_{i+1} \rangle} \underset{R}{:} \mathfrak{m}^\infty = Rf_{i+1} \supseteq 
I_{\langle d_i \rangle} \underset{R}{:} \mathfrak{m}^\infty= Rf_i
$$
for $i = 1, \ldots, r-1$.\\
(2) $\Rightarrow$ (1): This follows from Lemma \ref{lem_2-dim-prin-sdg}. 
\end{proof}

From Theorem \ref{thm_2-dim-prin-mdg} and  Theorem \ref{thm_2-dim-prin}, 
we obtain the following corollary.

\begin{corollary}\label{cor_2-dim-prin-mdg}
With the same conditions as in Theorem \ref{thm_2-dim-prin-mdg} (2), 
the following hold:
\begin{enumerate}[label=(\arabic*)]
\item
$\dim_K(I\otimes k) = d_1 - e_r + 1$ and $\dim_K\mathrm{soc}M = d_1 - e_r$.
\item
$\gamma\text{-}\mathrm{NZD}_R(M) = 
\hat{\gamma}\text{-}\mathrm{NZD}_R(M) = 
\mathrm{NZD}_R\left(\mathrm{C}^R_{\langle d_1 \rangle}M / \Gamma_{\mathfrak{m}}(\mathrm{C}^R_{\langle d_1 \rangle}M)\right)$.
\end{enumerate}
\end{corollary}

\begin{proof} We only prove (2). Since $M$ has a componetwise linear first syzygy, 
we obtain
$$
\emptyset\neq\gamma\text{-}\mathrm{NZD}_R(M) = 
\hat{\gamma}\text{-}\mathrm{NZD}_R(M) = 
\bigcap_{i=1}^r \mathrm{NZD}_R\left(\mathrm{C}^R_{\langle d_i \rangle}M / \Gamma_{\mathfrak{m}}(\mathrm{C}^R_{\langle d_i \rangle}M)\right)
$$ 
by  Remark \ref{rm_2-dim} (1) and Corollary \ref{cor_nzd} (3).\\
Moreover, noting that
$\mathrm{C}^R_{\langle d_i \rangle}M / \Gamma_{\mathfrak{m}}(\mathrm{C}^R_{\langle d_i \rangle}M) \simeq R / f_iR$ 
for $i = 1, \ldots, r$ and $f_{i+1} \mid f_i$ for $i = 1, \ldots, r-1$, we obtain
$$
\bigcap_{i=1}^r \mathrm{NZD}_R\left(\mathrm{C}^R_{\langle d_i \rangle}M / \Gamma_{\mathfrak{m}}(\mathrm{C}^R_{\langle d_i \rangle}M)\right) = 
\mathrm{NZD}_R\left(\mathrm{C}^R_{\langle d_1 \rangle}M / \Gamma_{\mathfrak{m}}(\mathrm{C}^R_{\langle d_1 \rangle}M)\right). 
$$
\end{proof}

%
%

\section*{Acknowledgement}
The author is deeply grateful to 
Professor Tadahito Harima and Professor Yuji Yoshino 
for their helpful comments and encouragement.
This work was supported by JSPS KAKENHI Grant Number JP20K03508 and JP19K03448.



%
%

\vspace{30pt}
\today
\end{document}